\documentclass[11pt, reqno]{amsart}
\usepackage{amssymb,latexsym,amsmath,amsfonts,mathdots,enumitem}
\usepackage{latexsym}
\usepackage[mathscr]{eucal}
\usepackage{colortbl,xcolor}
\usepackage{lmodern}
\usepackage{sansmathaccent}
\usepackage[latin1]{inputenc}
\usepackage{tikz}
\usepackage{physics}
\usetikzlibrary{shapes,arrows}
\usetikzlibrary{matrix,calc,shapes,arrows,positioning}
\pdfmapfile{+sansmathaccent.map}

\numberwithin{equation}{section}
\theoremstyle{plain}
\newtheorem{exam}{Example}

\newtheorem{thm}{Theorem}[section]

\newtheorem{cor}[thm]{Corollary}
\newtheorem{lem}[thm]{Lemma}
\newtheorem{prop}[thm]{Proposition}
\theoremstyle{definition}
\newtheorem{defn}[thm]{Definition}
\newtheorem{rem}[thm]{Remark}

\numberwithin{equation}{section}

\def\a{\alpha}

\def\beq{\begin{eqnarray}}
	\def\eeq{\end{eqnarray}}
\def\beqa{\begin{eqnarray*}}
	\def\eeqa{\end{eqnarray*}}

\def\beqn{\begin{equation}}
	\def\eeqn{\end{equation}}

\def\mg#1{}

\renewcommand{\epsilon}{\varepsilon}
\renewcommand{\phi}{\varphi}

\renewcommand{\sf}[1]{\textsf{#1}}

\numberwithin{equation}{section}
\allowdisplaybreaks[4] 

\setlist[enumerate]{font=\upshape,noitemsep, topsep=0pt} 
\setlist[itemize]{noitemsep, topsep=0pt}

\begin{document}
	
	\title{Interpolation by Rational $\Gamma_{E(3; 3; 1, 1, 1)}$-Inner Function and Rational $\Gamma_{E(3; 2; 1, 2)}$-Inner Function}
	\author{Avijit Pal,\hspace{0.1 cm} Bhaskar Paul}
	
	\address[A. Pal]{Department of Mathematics, IIT Bhilai, 6th Lane Road, Jevra, Chhattisgarh 491002}
	\email{A. Pal:avijit@iitbhilai.ac.in}
	
	\address[B. Paul]{Department of Mathematics, IIT Bhilai, 6th Lane Road, Jevra, Chhattisgarh 491002}
	\email{B. Paul:bhaskarpaul@iitbhilai.ac.in }
	
	
	
	\maketitle
	
	\begin{abstract}
We study several function-theoretic aspects of the domains
$G_{E(3;3;1,1,1)}$ and $G_{E(3;2;1,2)}$. We introduce and investigate
the notions of $\Gamma_{E(3;3;1,1,1)}$-inner functions and
$\Gamma_{E(3;2;1,2)}$-inner functions, and establish their relationships
with $\Gamma_{E(2;2;1,1)}$-inner functions and
$\Gamma_{E(2;1;2)}$-inner functions. We provide explicit constructions
and representations of rational
$\Gamma_{E(3;3;1,1,1)}$-inner functions and rational
$\Gamma_{E(3;2;1,2)}$-inner functions. As an application of these
results, we obtain partial solutions to the associated interpolation
problems and construct explicit rational inner interpolants for
interpolation data lying in suitable subsets of
$G_{E(3;3;1,1,1)}$ and $G_{E(3;2;1,2)}$, respectively.
\end{abstract}	
\section{Introduction}\label{Intro}

The \textit{Nevanlinna--Pick interpolation problem} is a classical and
active area of research in complex analysis. Its deep connections with
complex geometry and operator theory have led to the study of various
operator-theoretic and spectral versions of the interpolation problem.
In particular, the \textit{spectral Nevanlinna--Pick interpolation
problem} has attracted considerable attention from both operator
theorists and complex analysts. Given distinct points
\(
z_1,\ldots,z_m\in\mathbb D
\)
and matrices
\(
W_1,\ldots,W_m\in M_n(\mathbb C),
\)
the spectral Nevanlinna--Pick interpolation problem asks whether there
exists an analytic matrix-valued function
\(
F:\mathbb D\longrightarrow M_n(\mathbb C)
\)
such that
\[
F(z_i)=W_i,\quad 1\leq i\leq m,\quad\text{
and}\quad
r(F(z))\leq 1,\quad z\in\mathbb D,
\]
where $r(\cdot)$ denotes the spectral radius. While the classical
Nevanlinna--Pick interpolation problem is completely characterized by
the positivity of the associated Pick matrix, its spectral analogue is
considerably more subtle.
Let $M_n(\mathbb C)$ denote the algebra of all $n\times n$ complex
matrices, and let $E$ be a linear subspace of $M_n(\mathbb C)$. The
\textit{structured singular value} associated with $E$ is the function
\(
\mu_E:M_n(\mathbb C)\longrightarrow[0,\infty)
\)
defined by
\begin{equation}\label{mu}
\mu_E(A)
:=
\frac{1}{
\inf\{\|X\|:\det(I-AX)=0,\ X\in E\}
},
\qquad A\in M_n(\mathbb C),
\end{equation}
with the convention that $\mu_E(A)=0$ if $I-AX$ is nonsingular for
every $X\in E$ \cite{aj}. Throughout the paper, $\|\cdot\|$ denotes the
operator norm. For positive integers $r_1,\ldots,r_s$ satisfying
\(
r_1+\cdots+r_s=n,
\)
we consider the linear subspace
\begin{equation}\label{ls}
\begin{aligned}
E(n;s;r_1,\ldots,r_s)
:=
\Big\{
\operatorname{diag}
[z_1I_{r_1},\ldots,z_sI_{r_s}]
\in M_n(\mathbb C):
z_1,\ldots,z_s\in\mathbb C
\Big\}.
\end{aligned}
\end{equation}
In this paper, we are particularly interested in the structured
singular value domains associated with
\(
E(3;3;1,1,1),
E(3;2;1,2),
E(2;2;1,1),~\text{and}~
E(2;1;2).
\)
We recall the relevant closed domains
$\Gamma_{E(3;3;1,1,1)}$,
$\Gamma_{E(3;2;1,2)}$,
$\Gamma_{E(2;2;1,1)}$, and
$\Gamma_{E(2;1;2)}$;(see
\cite{Abouhajar,ay1,Bharali,apal1}).
Let $A_{(ij)}=\left(\begin{smallmatrix} a_{ii} & a_{ij}\\a_{ji} & a_{jj}\end{smallmatrix}\right).$ The domain associated with $E(2;2;1,1)$ is given by
\begin{equation*}
\begin{aligned}
\Gamma_{E(2;2;1,1)}
:=
\Big\{
(a_{11},a_{22},\det A)\in\mathbb C^3:
A=(a_{ij})_{i,j=1}^2\in M_2(\mathbb C),
\ \|A\|\leq1
\Big\}.
\end{aligned}
\end{equation*}
For $E(3;3;1,1,1)$, we have
\begin{equation*}
\begin{aligned}
\Gamma_{E(3;3;1,1,1)}
:=
\Big\{
\mathbf{x}=(a_{11},a_{22},\det A_{(12)},a_{33},
\det A_{(13)},\det A_{(23)},\det A)\in\mathbb C^7:
\\
\hspace{1cm}
A=(a_{ij})_{i,j=1}^3\in M_3(\mathbb C),
\quad
\mu_{E(3;3;1,1,1)}(A)\leq1
\Big\}.
\end{aligned}
\end{equation*}
Similarly,
\begin{equation*}
\begin{aligned}
\Gamma_{E(3;2;1,2)}
:=
\Big\{
(a_{11},\det A_{(12)}+\det A_{(13)},\det A,
a_{22}+a_{33},\det A_{(23)})\in\mathbb C^5:
\\
\hspace{1cm}
A\in M_3(\mathbb C),
\quad
\mu_{E(3;2;1,2)}(A)\leq1
\Big\},
\end{aligned}
\end{equation*}
and
\begin{equation*}
\begin{aligned}
\Gamma_{E(2;1;2)}
:=
\Big\{
(a_{11}+a_{22},\det A)\in\mathbb C^2:
A\in M_2(\mathbb C),
\ \|A\|\leq1
\Big\}.
\end{aligned}
\end{equation*}
The corresponding open domains will be denoted by
$G_{E(3;3;1,1,1)}$,
$G_{E(3;2;1,2)}$,
$G_{E(2;2;1,1)}$, and
$G_{E(2;1;2)}$, respectively.
The domains $G_{E(3;2;1,2)}$,
$G_{E(2;2;1,1)}$, and $G_{E(2;1;2)}$ are closely related to the
$\mu_{1,3}$-quotient, the tetrablock, and the symmetrized bidisc,
respectively; (see \cite{Abouhajar,ay1,Bharali}).

Several characterizations of
$G_{E(3;3;1,1,1)}$ were established in \cite{apal1}. In particular,
from the defining properties of $G_{E(3;3;1,1,1)}$, it follows that,
for every
\(
\mathbf{x}=(x_1,\ldots,x_7)\in G_{E(3;3;1,1,1)}
\)
and every $z,w\in\overline{\mathbb D}$,
\[
\begin{aligned}
1-zx_2-wx_4+zwx_6\neq0,\quad
1-zx_1-wx_4+zwx_5\neq0,\quad
1-zx_1-wx_2+zwx_3\neq0.
\end{aligned}
\]
Consequently, the following rational functions are well defined on
$\overline{\mathbb D}^{\,2}\times G_{E(3;3;1,1,1)}$:
\begin{equation}\label{Psi_1}
\begin{aligned}
\Psi^{(1)}(z_2,z_3,\mathbf{x})
&=
\frac{
x_1-z_2x_3-z_3x_5+z_2z_3x_7
}{
1-z_2x_2-z_3x_4+z_2z_3x_6
}.
\end{aligned}
\end{equation}
Similarly,
\begin{equation}\label{Psi_2}
\begin{aligned}
\Psi^{(2)}(z_1,z_3,\mathbf{x})
&=
\frac{
x_2-z_1x_3-z_3x_6+z_1z_3x_7
}{
1-z_1x_1-z_3x_4+z_1z_3x_5
}.
\end{aligned}
\end{equation}
Finally,
\begin{equation}\label{Psi_3}
\begin{aligned}
\Psi^{(3)}(z_1,z_2,\mathbf{x})
&=
\frac{
x_4-z_1x_5-z_2x_6+z_1z_2x_7
}{
1-z_1x_1-z_2x_2+z_1z_2x_3
}.
\end{aligned}
\end{equation}
The domain $G_{E(3;2;1,2)}$ was characterized in
\cite[Theorem 3.5]{Bharali}. In particular, if
\(
\widetilde{\mathbf{x}}
=
(x_1,x_2,x_3,y_1,y_2)
\in G_{E(3;2;1,2)},
\)
then
\(
y_2z^2-y_1z+1\neq0, z\in\overline{\mathbb D}.
\)
Therefore, the rational function
\begin{equation}\label{Psi}
\begin{aligned}
\Psi_3(z,\widetilde{\mathbf{x}})
&=
\frac{x_3z^2-x_2z+x_1}
     {y_2z^2-y_1z+1},
\quad z\in\overline{\mathbb D},
\end{aligned}
\end{equation}
is well defined.
We define the natural polynomial maps 
\begin{equation*}
\begin{aligned}
\Pi_{E(3;3;1,1,1)}:M_3(\mathbb C)
\longrightarrow\mathbb C^7\quad \text{and}\quad \Pi_{E(3;2;1,2)}:M_3(\mathbb C)&\longrightarrow\mathbb C^5
\end{aligned}
\end{equation*} by
\begin{equation*}
\begin{aligned}
\Pi_{E(3;3;1,1,1)}(A)
=
(a_{11},a_{22},\det A_{(12)},a_{33},
\det A_{(13)},\det A_{(23)},\det A)
\end{aligned}
\end{equation*}  and 
\begin{equation*}
\begin{aligned}
\Pi_{E(3;2;1,2)}(A)
=
(a_{11},\det A_{(12)}+\det A_{(13)},\det A,
a_{22}+a_{33},\det A_{(23)}),
\end{aligned}
\end{equation*} respectively.
Let $\mathcal U(3)$ denote the group of all $3\times3$ unitary
matrices. We define
\begin{equation*}
\begin{aligned}
K
:=
\Big\{
\mathbf{x}=(x_1,\ldots,x_7)\in
\Gamma_{E(3;3;1,1,1)}:
|x_7|=1,\quad
x_i=\overline{x}_{7-i}x_7,\quad 1\leq i\leq6
\Big\},
\end{aligned}
\end{equation*}
and
\begin{equation*}
\begin{aligned}
K_1
:=
\Big\{
\widetilde{\mathbf{x}}=(x_1,x_2,x_3,y_1,y_2)
\in\Gamma_{E(3;2;1,2)}:
|x_3|=1,\quad
x_1=\overline{y}_2x_3,\quad
x_2=\overline{y}_1x_3
\Big\}.
\end{aligned}
\end{equation*}
We shall also use the open spectral ball
\begin{equation}\label{Omega_n}
\begin{aligned}
\Omega_n
:=
\{A\in M_n(\mathbb C):r(A)<1\}.
\end{aligned}
\end{equation}
Since
\(
r(A)\leq\|A\|,
\)
we have
\begin{equation}\label{Omega_n in B_MnC}
B_{M_n(\mathbb C)}(0,1)\subseteq\Omega_n,
\end{equation}
where $B_{M_n(\mathbb C)}(0,1)$ denotes the open unit ball of
$M_n(\mathbb C)$ with respect to the operator norm. In this paper, we
shall primarily work with $\Omega_3$.
We begin by recalling
some standard notation for reflected polynomial.
Let $f$ be a polynomial of degree at most $n$, where $n\geq 0$. We define
the \emph{reflected polynomial} $f^{\sim n}$ by
\begin{equation}\label{f sim}
\begin{aligned}
f^{\sim n}(z)
&=
z^n\overline{f(1/\overline{z})}.
\end{aligned}
\end{equation}
We also define
\begin{equation}\label{f vee}
\begin{aligned}
f^{\vee}(z)
&=
\overline{f(\overline{z})}.
\end{aligned}
\end{equation}
It follows immediately from the definitions that
\begin{equation}\label{f sim, f vee}
\begin{aligned}
f^{\sim n}(z)
=
z^nf^{\vee}(1/z).
\end{aligned}
\end{equation}
Recall that a \textit{classical inner function} on $\mathbb D$ is a
holomorphic function which maps $\mathbb D$ into $\mathbb D$ and has
radial boundary values of modulus one almost everywhere on
$\mathbb T$. A rational inner function on $\mathbb D$ of degree at
most $n$ is a finite Blaschke product and, equivalently, can be written
in the form
\begin{equation}\label{Description of inner function}
f(z)
=
\frac{D^{\sim n}(z)}{D(z)},
\end{equation}
where $D$ is a polynomial of degree at most $n$ having no zeros in
$\overline{\mathbb D}$; (see, for example,
\cite[Corollary 6.10]{JAZLNY 1}).
Further details concerning inner functions associated with
$\Gamma_{E(2;1;2)}$ and $\Gamma_{E(2;2;1,1)}$ can be found in
\cite{JAZLNY,Alsalhi}. Motivated by these notions, we introduce
$\Gamma_{E(3;3;1,1,1)}$-inner functions and
$\Gamma_{E(3;2;1,2)}$-inner functions and develop their basic
function-theoretic properties.
The motivation for this work comes from interpolation theory. In the
scalar setting, finite solvable interpolation problems on the unit
disc admit finite Blaschke products as interpolants. In the
matrix-valued setting, rational matrix inner functions play an
analogous role. McCarthy and Young \cite{McCarthy 1} established the
existence of rational inner interpolants whenever a matrix-valued
interpolation problem admits a contractive analytic solution.
Costara \cite{costara} established an important connection between
classical matrix-valued interpolation problems and spectral
Nevanlinna--Pick interpolation problems through similarity
transformations. These developments motivate the study of inner
functions associated with the domains considered in this paper.
We now introduce the notions of
$\Gamma_{E(3;3;1,1,1)}$-inner functions and
$\Gamma_{E(3;2;1,2)}$-inner functions.
\begin{defn}\label{Defn 1}
\begin{enumerate}
\item
A holomorphic function
\(
\mathbf{x}=(x_1,\ldots,x_7):
\mathbb D\longrightarrow
\Gamma_{E(3;3;1,1,1)}
\)
is called a
\textit{$\Gamma_{E(3;3;1,1,1)}$-inner function} if its radial limit
\begin{equation}\label{Gamma_{E(3;3;1,1,1)}-inner function}
\lim_{r\to1^-}\mathbf{x}(rz)
\end{equation}
belongs to $K$ for almost every $z\in\mathbb T$ with respect to
Lebesgue measure.

\item
A holomorphic function
\(
\widetilde{\mathbf{x}}
=
(x_1,x_2,x_3,y_1,y_2):
\mathbb D\longrightarrow
\Gamma_{E(3;2;1,2)}
\)
is called a
\textit{$\Gamma_{E(3;2;1,2)}$-inner function} if its radial limit
\begin{equation}\label{Gamma_{E(3;2;1,2)}-inner function}
\lim_{r\to1^-}\widetilde{\mathbf{x}}(rz)
\end{equation}
belongs to $K_1$ for almost every $z\in\mathbb T$ with respect to
Lebesgue measure.
\end{enumerate}
\end{defn}

By Fatou's theorem, the radial limits in
\eqref{Gamma_{E(3;3;1,1,1)}-inner function} and
\eqref{Gamma_{E(3;2;1,2)}-inner function} exist almost everywhere on
$\mathbb T$ for bounded holomorphic coordinate functions. It is useful
to note that if $\mathbf{x}$ is a rational
$\Gamma_{E(3;3;1,1,1)}$-inner function, then each coordinate
$x_1,\ldots,x_7$ is holomorphic on $\mathbb D$ and has no poles in
$\overline{\mathbb D}$. Hence all the coordinate functions extend
continuously to $\overline{\mathbb D}$. Similarly, if
$\widetilde{\mathbf{x}}$ is a rational
$\Gamma_{E(3;2;1,2)}$-inner function, then the coordinate functions
$x_1,x_2,x_3,y_1,y_2$ have no poles in $\overline{\mathbb D}$ and
therefore extend continuously to $\overline{\mathbb D}$.

We conclude this introductory discussion with an outline of the
paper. In Section~\ref{Properties}, we establish various properties
of $\Gamma_{E(3;3;1,1,1)}$-inner functions and
$\Gamma_{E(3;2;1,2)}$-inner functions. Section~\ref{Degree} is devoted
to determining their degrees. In Section~\ref{Relation}, we investigate
their relationships with $\Gamma_{E(2;2;1,1)}$-inner functions and
$\Gamma_{E(2;1;2)}$-inner functions, as well as the relationships
between the two classes of inner functions introduced above.
Section~\ref{Example} contains several concrete examples of
$\Gamma_{E(3;3;1,1,1)}$-inner functions and
$\Gamma_{E(3;2;1,2)}$-inner functions. In Section~\ref{Description},
we obtain explicit descriptions of rational
$\Gamma_{E(3;3;1,1,1)}$-inner functions and rational
$\Gamma_{E(3;2;1,2)}$-inner functions. Finally, in
Section~\ref{Interpolation}, we obtain partial solutions to the
corresponding interpolation problems and construct rational
$\Gamma_{E(3;3;1,1,1)}$-inner and
$\Gamma_{E(3;2;1,2)}$-inner interpolants for suitable classes of
interpolation data in
$G_{E(3;3;1,1,1)}$ and $G_{E(3;2;1,2)}$, respectively.	
	
\section{Properties of $\Gamma_{E(3;3;1,1,1)}$-Inner Functions and
$\Gamma_{E(3;2;1,2)}$-Inner Functions}
\label{Properties}

We begin by characterizing the  sets
$K$ and $K_1$.  The
following characterization will be useful in describing the boundary
behavior of $\Gamma_{E(3;3;1,1,1)}$-inner functions.

\begin{prop}\label{Prop K}
	Let $\mathbf{x}=(x_1,\ldots,x_7)\in\mathbb{C}^7$. Then the following
	conditions are equivalent:
	\begin{enumerate}
		\item $\mathbf{x}\in K$;
		\item $\mathbf{x}\in\Gamma_{E(3;3;1,1,1)}$ and $|x_7|=1$.
	\end{enumerate}
\end{prop}

\begin{proof}
	$(1)\Rightarrow(2)$.
	By the definition of $K$, we have
	\(
	\mathbf{x}\in\Gamma_{E(3;3;1,1,1)}~\text{and}~
	|x_7|=1.
	\)
	Thus $(2)$ follows immediately.

	$(2)\Rightarrow(1)$.
	Suppose that
	\(
	\mathbf{x}\in\Gamma_{E(3;3;1,1,1)}~\text{and}~
	|x_7|=1.
	\)
	By \cite[Theorem 2.29]{apal1}, we have
	\[
	(x_i,x_{7-i},x_7)\in\Gamma_{E(2;2;1,1)},
	\quad 1\leq i\leq 6.
	\]
	Since $|x_7|=1$, it follows from \cite[Theorem 7.1]{Abouhajar}
	that
	\[
	(x_i,x_{7-i},x_7)
	\in b\Gamma_{E(2;2;1,1)},
	\quad 1\leq i\leq 6.
	\]
	By the characterization of the distinguished boundary of
	$\Gamma_{E(2;2;1,1)}$, this implies
	\(
	x_i=\overline{x_{7-i}}x_7, 1\leq i\leq 6.
	\)
	Consequently, $\mathbf{x}\in K$, by the definition of $K$.
\end{proof}

The corresponding characterization for $K_1$ follows by an analogous
argument. We record it below for later use.

\begin{prop}\label{Prop K_1}
	Let
	\(
	\widetilde{\mathbf{x}}
	=(x_1,x_2,x_3,y_1,y_2)\in\mathbb{C}^5.
	\)
	Then the following conditions are equivalent:
	\begin{enumerate}
		\item $\widetilde{\mathbf{x}}\in K_1$;
		\item $\widetilde{\mathbf{x}}\in\Gamma_{E(3;2;1,2)}$ and
		$|x_3|=1$.
	\end{enumerate}
\end{prop}

We now investigate several properties of
$\Gamma_{E(3;3;1,1,1)}$-inner functions and
$\Gamma_{E(3;2;1,2)}$-inner functions. In particular, we describe the
boundary behavior of their component functions and, in the rational
case, obtain relations between the zeros and poles of the component
functions.

\begin{lem}\label{Lem 1}
	Let
	\(
	\mathbf{x}=(x_1,\ldots,x_7)
	\)
	be a $\Gamma_{E(3;3;1,1,1)}$-inner function. Then the following
	statements hold:
	\begin{enumerate}
		\item For almost every $z\in\mathbb{T}$,
		\(
		x_i(z)
		=
		\overline{x_{7-i}(z)}\,x_7(z), 1\leq i\leq 6,
		\)
		and
		\[
		|x_i(z)|\leq 1,\quad 1\leq i\leq 6,
		\quad
		|x_7(z)|=1.
		\]
		
		\item The function $x_7$ is an inner function.
	\end{enumerate}
\end{lem}

\begin{proof}
	(1) Since $\mathbf{x}$ is a
	$\Gamma_{E(3;3;1,1,1)}$-inner function, by definition,
	\(
	\mathbf{x}(z)\in K
	\)
	for almost every $z\in\mathbb{T}$. By Proposition
	\ref{Prop K}, this is equivalent to
	\[
	\mathbf{x}(z)\in
	\Gamma_{E(3;3;1,1,1)}
	\quad\text{and}\quad
	|x_7(z)|=1
	\]
	for almost every $z\in\mathbb{T}$. Moreover, the  definition of $K$ give
	\(
	x_i(z)
	=
	\overline{x_{7-i}(z)}\,x_7(z), 1\leq i\leq 6,
	\)
	and
	\[
	|x_i(z)|\leq 1,
	\quad 1\leq i\leq 6.
	\]
	
	(2) Since
	\(
	x_7:\mathbb{D}\longrightarrow\overline{\mathbb{D}}
	\)
	is holomorphic and its radial boundary values satisfy
	\(
	|x_7(z)|=1
	\)
	for almost every $z\in\mathbb{T}$, it follows that $x_7$ is an
	inner function.
\end{proof}

As an immediate consequence of Lemma \ref{Lem 1}, we obtain the
following boundary property of the associated functions
$\Psi^{(i)}$.

\begin{cor}\label{Cor 1}
	Let
	\(
	\mathbf{x}=(x_1,\ldots,x_7)
	\)
	be a $\Gamma_{E(3;3;1,1,1)}$-inner function. Then, for almost every
	$z\in\mathbb{T}$,
	\[
	\left\|
	\Psi^{(i)}(\,\cdot\,,\mathbf{x}(z))
	\right\|_{H^\infty(\mathbb{D}^2)}
	=1,
	\quad 1\leq i\leq 3.
	\]
\end{cor}We next consider the additional structure that arises when the
$\Gamma_{E(3;3;1,1,1)}$-inner function is rational.

\begin{lem}\label{Lem 2}
	Let
	\(
	\mathbf{x}=(x_1,\ldots,x_7)
	\)
	be a rational $\Gamma_{E(3;3;1,1,1)}$-inner function. Then
	\[
	x_i(z)
	=
	x_{7-i}^{\vee}(1/z)\,x_7(z), \quad1\leq i\leq 6,
	\]
	for every $z\in\mathbb{C}$ for which both sides are defined.
\end{lem}

\begin{proof}
	Since $\mathbf{x}$ is a
	$\Gamma_{E(3;3;1,1,1)}$-inner function, Lemma \ref{Lem 1} yields
	\(
	x_i(z)
	=
	\overline{x_{7-i}(z)}\,x_7(z),1\leq i\leq 6,
	\)
	for almost every $z\in\mathbb{T}$.
Recall that, for a rational function $f$, its reversed function is
	defined by
	\[
	f^\vee(z)=\overline{f(1/\overline z)}.
	\]
	Hence, for $z\in\mathbb{T}$, we have
	\[
	\overline{f(z)}
	=
	f^\vee(1/z).
	\]
	In particular,
	\[
	\overline{x_{7-i}(z)}
	=
	x_{7-i}^\vee(1/z),
	\quad z\in\mathbb{T},
	\]
	and therefore
	\[
	x_i(z)
	=
	x_{7-i}^\vee(1/z)x_7(z),
	\quad 1\leq i\leq 6,
	\]
	for almost every $z\in\mathbb{T}$.
Now, since each $x_j$ is rational, the functions
	\(
	x_i(z)~\text{and}~
	x_{7-i}^\vee(1/z)x_7(z)
	\)
	are rational functions of $z$. Their difference is therefore a
	rational function which vanishes for almost every $z\in\mathbb{T}$.
	In particular, it has infinitely many zeros with an accumulation
	point in $\mathbb{C}$ away from its poles. Hence, by the identity
	theorem for rational functions, this difference is identically zero.
	Consequently,
	\(
	x_i(z)
	=
	x_{7-i}^\vee(1/z)x_7(z),1\leq i\leq 6.
	\)
	This proves the assertion.
\end{proof}	
For a meromorphic function $f$ on $\mathbb D$ and a point
$a\in\mathbb D$, we denote by $\operatorname{ord}_{a}(f)$ the order of
$f$ at $a$. More precisely, if $f$ has a zero of order $m$ at $a$, then
\(
\operatorname{ord}_{a}(f)=m,
\)
whereas if $f$ has a pole of order $m$ at $a$, then
\(
\operatorname{ord}_{a}(f)=-m.
\)
If $f$ is holomorphic and nonzero at $a$, we set
\(
\operatorname{ord}_{a}(f)=0.
\)

	\begin{prop}\label{Prop 1}
	Let
	\(
	\mathbf{x}=(x_1,\ldots,x_7)
	\)
	be a rational $\Gamma_{E(3;3;1,1,1)}$-inner function. Then the
	following statements hold.
	\begin{enumerate}
		\item 
		\begin{enumerate}
		\item Let $a_j\in\mathbb{C}$ be a pole of $x_7$ of multiplicity
		$k_j\geq 1$, and suppose that
		$\frac{1}{\overline{a_j}}$ is a zero of $x_{7-j}$ of multiplicity
		$l_j\geq 0$. Then $a_j$ is a pole of $x_j$ of multiplicity at
		least
		\(
		k_j-l_j,
 1\leq j\leq 3,
		\)
		whenever $k_j-l_j>0$.
		\item If $\infty$ is a pole of $x_7$ of multiplicity $k\geq 1$ and
		$0$ is a zero of $x_{7-j}$ of multiplicity $l_j\geq 0$, then
		$\infty$ is a pole of $x_j$ of multiplicity at least
		\(
		k-l_j, 1\leq j\leq 3,
		\)
		whenever $k-l_j>0$.
		\end{enumerate}
		
		\item \begin{enumerate}
		\item If $a_j\in\mathbb{C}$ is a pole of $x_j$ of multiplicity
		$k_j\geq 1$, then $a_j$ is a pole of $x_7$ of multiplicity at
		least $k_j$, for $1\leq j\leq 3$.
	\item If $\infty$ is a pole of $x_j$ of multiplicity $k_j$, then
		$\infty$ is a pole of $x_7$ of multiplicity at least
		\(
		\max\{k_1,k_2,k_3\}.
		\)
	\end{enumerate}
	\end{enumerate}
\end{prop}

\begin{proof}
	By Lemma~\ref{Lem 2}, for $1\leq i\leq 6$ we have the identity of
	rational functions
	\begin{equation}\label{P 3}
		x_i(z)=x_{7-i}^{\vee}(1/z)x_7(z).
	\end{equation}
We first prove (1)(a). Since $x_7$ is a rational inner function on
	$\mathbb{D}$, it has no poles in $\overline{\mathbb{D}}$. Hence
	every finite pole $a_j$ of $x_7$ satisfies
	\(
	|a_j|>1.
	\)
	Consequently,
	\(
	\left|\frac{1}{a_j}\right|<1.
	\)
Fix $1\leq j\leq 3$. Suppose that $a_j$ is a pole of $x_7$ of
	multiplicity $k_j$, and that
	$\frac{1}{\overline{a_j}}$ is a zero of $x_{7-j}$ of multiplicity
	$l_j$.
Recall that
	\[
	x_{7-j}^{\vee}(w)
	=
	\overline{x_{7-j}(1/w)}.
	\]
	Thus the zero of $x_{7-j}$ at
	$\frac{1}{\overline{a_j}}$ of multiplicity $l_j$ corresponds to a
	zero of $x_{7-j}^{\vee}$ at $a_j$ of multiplicity $l_j$. Hence
	there exists a holomorphic function $g_j$ in a neighbourhood of
	$a_j$, with
	\(
	g_j(a_j)\neq0,
	\)
	such that
	\[
	x_{7-j}^{\vee}(z)
	=
	(z-a_j)^{l_j}g_j(z).
	\]
	On the other hand, since $a_j$ is a pole of $x_7$ of multiplicity
	$k_j$, there exists a holomorphic function $h_j$ near $a_j$,
	with $h_j(a_j)\neq0$, such that
	\[
	x_7(z)
	=
	(z-a_j)^{-k_j}h_j(z).
	\]
	Substituting these expressions into \eqref{P 3}, we obtain
	\[
	x_j(z)
	=
	(z-a_j)^{l_j-k_j}g_j(z)h_j(z).
	\]
	Therefore, if $k_j>l_j$, then $x_j$ has a pole at $a_j$ of
	multiplicity exactly $k_j-l_j$, and in particular of multiplicity
	at least $k_j-l_j$.

	We next consider the case $(1)(b).$  Suppose that $\infty$ is a pole of $x_7$.
	Let
	\(
	w=\frac{1}{z}.
	\)
	Then $\infty$ is a pole of $x_7$ of multiplicity $k$ if and only if
	$0$ is a pole of $x_7(1/w)$ of multiplicity $k$. From
	\eqref{P 3}, with $z$ replaced by $1/w$, we obtain
	\begin{equation}\label{P 4}
		x_j(1/w)
	=
	x_{7-j}^{\vee}(w)x_7(1/w).
	\end{equation}
	
	Since $0$ is a zero of $x_{7-j}$ of multiplicity $l_j$, the
	definition of the reversed function implies that
	$x_{7-j}^{\vee}$ has a zero at $\infty$ of multiplicity $l_j$,
	or equivalently, in the local coordinate $w=1/z$, the factor
	$x_{7-j}^{\vee}(w)$ has a zero at $w=0$ of multiplicity $l_j$.
	Thus
	\[
	x_{7-j}^{\vee}(w)=w^{l_j}g_j(w),
	\qquad g_j(0)\neq0,
	\]
	while
	\[
	x_7(1/w)=w^{-k}h(w),
	\qquad h(0)\neq0.
	\]
	Consequently, \eqref{P 4} gives
	\[
	x_j(1/w)
	=
	w^{l_j-k}g_j(w)h(w).
	\]
	Therefore, if $k>l_j$, then $x_j(1/w)$ has a pole at $w=0$ of
	multiplicity $k-l_j$. Equivalently, $x_j$ has a pole at $\infty$
	of multiplicity at least $k-l_j$. This proves (1).

	We now prove (2)(a). Let $a_j\in\mathbb{C}$ be a pole of $x_j$ of
	multiplicity $k_j$. Since $x_j$ is holomorphic on $\mathbb{D}$,
	we necessarily have
	\(
	|a_j|>1.
	\)
	Hence
	\(
	\left|\frac{1}{a_j}\right|<1.
	\)
	In particular, $x_{7-j}^{\vee}(1/z)$ is holomorphic and finite in a
	neighbourhood of $a_j$. More precisely, there exists a holomorphic
	function $g_j$ near $a_j$ such that
	\[
	x_{7-j}^{\vee}(1/z)=g_j(z).
	\]
	By \eqref{P 3},
	\[
	x_j(z)=g_j(z)x_7(z).
	\]
	Since $x_j$ has a pole at $a_j$ of multiplicity $k_j$ while
	$g_j$ is holomorphic at $a_j$, the function $x_7$ must have a pole
	at $a_j$ of multiplicity at least $k_j$. Thus
	\[
	\operatorname{ord}_{a_j}(x_7)\leq-k_j,
	\]
	which proves the first assertion in (2)(a).

	Finally we prove $(2)(b).$  Suppose that $\infty$ is a pole of $x_j$ of multiplicity
	$k_j$. Again put $w=1/z$. Then
	$x_j(1/w)$ has a pole at $w=0$ of multiplicity $k_j$. By
	\eqref{P 4},
	\(
	x_j(1/w)
	=
	x_{7-j}^{\vee}(w)x_7(1/w).
	\)
	Since $x_{7-j}$ is holomorphic at $0$, the function
	$x_{7-j}^{\vee}(w)$ is holomorphic at $w=0$. Hence the pole of
	$x_j(1/w)$ at $w=0$ must come entirely from $x_7(1/w)$.
	Therefore $x_7(1/w)$ has a pole at $w=0$ of multiplicity at least
	$k_j$. Equivalently, $x_7$ has a pole at $\infty$ of multiplicity
	at least $k_j$.
This holds for each $j=1,2,3$. Consequently, the multiplicity of
	the pole of $x_7$ at $\infty$ is at least
	\(
	\max\{k_1,k_2,k_3\}.
	\)
	This completes the proof.
\end{proof}

In the remainder of this section, we establish analogous results
for $\Gamma_{E(3;2;1,2)}$-inner functions. The proofs follow the same
arguments as those of Lemma~\ref{Lem 1}, Corollary~\ref{Cor 1},
Lemma~\ref{Lem 2}, and Proposition~\ref{Prop 1}, respectively, and
are therefore omitted.

\begin{lem}\label{Lem 3}
	Let
	\(
	\widetilde{\mathbf{x}}
	=(x_1,x_2,x_3,y_1,y_2)
	\)
	be a $\Gamma_{E(3;2;1,2)}$-inner function. Then the following
	statements hold:
	\begin{enumerate}
		\item For almost every $z\in\mathbb{T}$,
		\[
		x_1(z)=\overline{y_2(z)}x_3(z),
		\quad
		x_2(z)=\overline{y_1(z)}x_3(z),
		\quad
		|x_3(z)|=1.
		\]
		
		\item $x_3$ is an inner function.
	\end{enumerate}
\end{lem}

The following consequence of Lemma~\ref{Lem 3} will be useful.

\begin{cor}\label{Cor 2}
	Let
	\(
	\widetilde{\mathbf{x}}
	=(x_1,x_2,x_3,y_1,y_2)
	\)
	be a $\Gamma_{E(3;2;1,2)}$-inner function. Then, for almost every
	$z\in\mathbb{T}$,
	\[
	\left\|
	\Psi_3(\,\cdot\,,\widetilde{\mathbf{x}}(z))
	\right\|_{H^\infty(\mathbb{D})}
	=1.
	\]
\end{cor}

\begin{lem}\label{Lem 4}
	Let
	\(
	\widetilde{\mathbf{x}}
	=(x_1,x_2,x_3,y_1,y_2)
	\)
	be a rational $\Gamma_{E(3;2;1,2)}$-inner function. Then the
	following identities hold as identities of rational functions:
	\[
		x_1(z)=y_2^\vee(1/z)x_3(z),
		\quad
		x_2(z)=y_1^\vee(1/z)x_3(z).
	\]
\end{lem}

\begin{prop}\label{Prop 2}
	Let
	\(
	\widetilde{\mathbf{x}}
	=(x_1,x_2,x_3,y_1,y_2)
	\)
	be a rational $\Gamma_{E(3;2;1,2)}$-inner function. Then the
	following statements hold:
	\begin{enumerate}
		\item \begin{enumerate}
		\item Suppose that $a_1,a_2\in\mathbb{C}$ are poles of $x_3$ of
		multiplicities $k_1,k_2\geq1$, respectively, and that
		\(
		\frac{1}{\overline{a_1}}~\text{and}~
		\frac{1}{\overline{a_2}}
		\)
		are zeros of $y_2$ and $y_1$ of multiplicities $l_1$ and $l_2$,
		respectively. Then $a_1$ and $a_2$ are poles of $x_1$ and $x_2$,
		respectively, of multiplicities at least
		\(
		k_1-l_1
		~\text{and}~
		k_2-l_2,
		\)
		whenever these quantities are positive.
		
		\item If $\infty$ is a pole of $x_3$ of multiplicity $k\geq1$ and
		$0$ is a zero of $y_2$ and $y_1$ of multiplicities $l_1$ and
		$l_2$, respectively, then $\infty$ is a pole of $x_1$ and
		$x_2$ of multiplicities at least
		\(
		k-l_1
		~\text{and}~
		k-l_2,
		\)
		respectively, whenever these quantities are positive.
		\end{enumerate}
		
		\item \begin{enumerate}
		\item If $a_1,a_2\in\mathbb{C}$ are poles of $x_1$ and $x_2$
		of multiplicities $k_1,k_2\geq1$, respectively, then $a_1$ and
		$a_2$ are poles of $x_3$ of multiplicities at least $k_1$ and
		$k_2$, respectively.
		
		\item If $\infty$ is a pole of $x_1$ and $x_2$ of multiplicities
		$k_1$ and $k_2$, respectively, then $\infty$ is a pole of
		$x_3$ of multiplicity at least
		\(
		\max\{k_1,k_2\}.
		\)\end{enumerate}
	\end{enumerate}
\end{prop}	
The relations established in Propositions~\ref{Prop 1} and
\ref{Prop 2} yield several further consequences concerning the
reciprocal structure of rational inner functions.
For $a\in\mathbb{C}\setminus\{0\}$, define
\[
a^*:=\frac{1}{\overline{a}}.
\]
Notice that the map $a\mapsto a^*$ is an involution, that is,
\(
(a^*)^*=a.
\)

\begin{cor}\label{cor:reciprocal-pairing}
	Let
	\(
	\mathbf{x}=(x_1,\ldots,x_7)
	\)
	be a rational $\Gamma_{E(3;3;1,1,1)}$-inner function. Suppose that
	$a\in\mathbb{C}\setminus\{0\}$ is neither a zero nor a pole of
	$x_7$. Then, for $1\leq i\leq3$,
	\[
	\operatorname{ord}_{a}(x_i)
	=
	\operatorname{ord}_{a^*}(x_{7-i}).
	\]
	Consequently, zeros and poles of $x_i$ and $x_{7-i}$ occur in
	reciprocal-conjugate pairs and have the same multiplicities.
\end{cor}

\begin{proof}
	By Lemma~\ref{Lem 2},
	\(
	x_i(z)=x_{7-i}^{\vee}(1/z)x_7(z).
	\)
	Since $a$ is neither a zero nor a pole of $x_7$,
	\(
	\operatorname{ord}_{a}(x_7)=0.
	\)
	Taking orders at $a$ and using the reversal property gives
	\(
	\operatorname{ord}_{a}(x_i)
	=
	\operatorname{ord}_{a^*}(x_{7-i}).
	\)
	The final assertion follows immediately.
\end{proof}

\begin{cor}\label{cor:reciprocal-pairing-Theta}
	Let
	\(
	\widetilde{\mathbf{x}}
	=(x_1,x_2,x_3,y_1,y_2)
	\)
	be a rational $\Gamma_{E(3;2;1,2)}$-inner function. Suppose that
	$a\in\mathbb{C}\setminus\{0\}$ is neither a zero nor a pole of
	$x_3$. Then
\[
	\operatorname{ord}_{a}(x_1)
	=
	\operatorname{ord}_{a^*}(y_2)\quad \text{
	and}\quad 
	\operatorname{ord}_{a}(x_2)
	=
	\operatorname{ord}_{a^*}(y_1).
	\]
	Hence the zeros and poles of $x_1$ and $y_2$, as well as those of
	$x_2$ and $y_1$, occur in reciprocal-conjugate pairs with the same
	multiplicities.
\end{cor}

\begin{proof}
	This follows directly from Proposition~\ref{Prop 2} and the fact
	that $a$ is neither a zero nor a pole of $x_3$. This completes the proof.
\end{proof}

\begin{prop}\label{prop:reciprocal-orbits}
	Let $\mathbf{x}$ be a rational
	$\Gamma_{E(3;3;1,1,1)}$-inner function. Suppose that $a\neq0$ is either
	a zero or a  pole of $x_i$, $1\leq i\leq3$, and that $a$ is neither a
	zero nor a pole of $x_7$. Then $a^*$ is  either a zero or a pole, respectively,
	of $x_{7-i}$ with the same multiplicity. Moreover,
	\(
	(a^*)^*=a.
	\)
In particular,  the corresponding zeros and poles occur in reciprocal pairs.
\end{prop}

\begin{proof}
	The first assertion follows from Corollary~\ref{cor:reciprocal-pairing}.
	The second assertion follows from
	\[
	(a^*)^*
	=
	\frac{1}{\overline{1/\overline a}}
	=a.
	\] This completes the proof.
\end{proof}

\begin{prop}\label{prop:degree-estimate}
	Let
	\(
	\mathbf{x}=(x_1,\ldots,x_7)
	\)
	be a rational $\Gamma_{E(3;3;1,1,1)}$-inner function. Then
	\[
	\deg x_i
	\leq
	\deg x_{7-i}+\deg x_7,
	\quad 1\leq i\leq3.
	\]
	Similarly, if
	\(
	\widetilde{\mathbf{x}}
	=(x_1,x_2,x_3,y_1,y_2)
	\)
	is a rational $\Gamma_{E(3;2;1,2)}$-inner function, then
	\[
	\deg x_1\leq\deg y_2+\deg x_3~\text{
	and}~
	\deg x_2\leq\deg y_1+\deg x_3.
	\]
\end{prop}

\begin{proof}
	By Lemma~\ref{Lem 2},
	\(
	x_i(z)=x_{7-i}^{\vee}(1/z)x_7(z),1\leq i\leq3.
	\)
	Since reversal and the transformation $z\mapsto1/z$ preserve the
	degree of a rational function,
	\(
	\deg\big(x_{7-i}^{\vee}(1/z)\big)=\deg x_{7-i}.
	\)
	Therefore
	\[
	\deg x_i
	\leq
	\deg x_{7-i}+\deg x_7.
	\]
	The second assertion follows similarly from Lemma~\ref{Lem 4}:
	\[
	x_1(z)=y_2^\vee(1/z)x_3(z),
	\quad
	x_2(z)=y_1^\vee(1/z)x_3(z).
	\] This completes the proof.
\end{proof}

\begin{rem}\label{rem:fixed-points}
	The involution
	\(
	a\longmapsto a^*=\frac{1}{\overline a}
	\)
	has as its fixed-point set precisely the unit circle:
	\[
	a^*=a
	\quad\Longleftrightarrow\quad
	|a|=1.
	\]
	Thus the reciprocal pairing described above collapses on
	$\mathbb{T}$. This is consistent with the boundary relations for
	inner functions. In particular, for almost every $z\in\mathbb{T}$,
	\[
	x_i(z)=\overline{x_{7-i}(z)}x_7(z),
	\quad 1\leq i\leq3,
	\]
	with $|x_7(z)|=1$, while
	\[
	x_1(z)=\overline{y_2(z)}x_3(z),
	\quad
	x_2(z)=\overline{y_1(z)}x_3(z),
	\]
	with $|x_3(z)|=1$.
Propositions~\ref{Prop 1} and~\ref{Prop 2}, together with the
	preceding corollaries, show that the components of a rational
	$\Gamma_{E(3;3;1,1,1)}$-inner function and a rational
	$\Gamma_{E(3;2;1,2)}$-inner function are subject to a strong
	reciprocal  structure. In particular, the map
	\(
	a\longmapsto\frac{1}{\overline a}
	\)
	provides the natural symmetry relating the corresponding component
	functions. The distinguished components $x_7$ and $x_3$ determine
	the possible deviation from an exact reciprocal correspondence.
\end{rem}	
\section{Degrees of Rational $\Gamma_{E(3;3;1,1,1)}$-Inner Functions and
$\Gamma_{E(3;2;1,2)}$-Inner Functions}
\label{Degree}
Let $f$ be a nonconstant rational function represented in reduced form
as
\[
f=\frac{p}{q},
\quad
p,q\in\mathbb{C}[z],
\quad
\gcd(p,q)=1.
\]
We define the degree of $f$ by
\[
\deg(f):=\max\{\deg p,\deg q\}.
\]
Equivalently, $\deg(f)$ is the degree of $f$ as a holomorphic map
from the Riemann sphere $\widehat{\mathbb{C}}$ onto itself.
In this section, we introduce two natural notions of degree for
rational $\Gamma_{E(3;3;1,1,1)}$-inner functions and rational
$\Gamma_{E(3;2;1,2)}$-inner functions. The first notion is defined
in terms of the dimension of an associated model space in the Hardy
space. The second is defined topologically using the homomorphism
induced on the fundamental groups. We shall show that these two
notions coincide.
Let $H^2$ denote the Hardy space on the unit circle $\mathbb{T}$.
We use the notation
\begin{equation}\label{H^2bar}
	\overline{H^2}
	:=
	\{f\in L^2(\mathbb{T}):\overline{f}\in H^2\}
\end{equation}
and
\begin{equation}\label{H2minus}
	H^2_-
	:=
	\left\{
	f\in L^2(\mathbb{T}):
	\widehat f(m)=0
	\text{ for all }m\geq0
	\right\}.
\end{equation}
It is well known that
\(
\overline{H^2}
=
zH^2_-,
\)
where $z=\operatorname{id}_{\mathbb D}$ denotes the identity function
on $\mathbb D$.
We first consider rational
$\Gamma_{E(3;3;1,1,1)}$-inner functions.
Let
\(
\mathbf{x}=(x_1,\ldots,x_7)
\)
be a rational $\Gamma_{E(3;3;1,1,1)}$-inner function. By Lemma~\ref{Lem 1},
the function $x_7$ is an inner function. Since $x_7$ is rational, it is
a finite Blaschke product. Suppose that
\(
\deg(x_7)=n.
\)
Furthermore, since $\mathbf{x}$ is
$\Gamma_{E(3;3;1,1,1)}$-inner,
\[
\mathbf{x}(z)\in K
\quad\text{for almost every }z\in\mathbb{T}.
\]
In particular,
\(
x_i\in H^\infty(\mathbb D)\subset H^2, 1\leq i\leq7.
\)
By Lemma~\ref{Lem 1}, for almost every $z\in\mathbb{T}$,
\begin{equation}\label{eq:boundary-relation-degree}
	x_i(z)
	=
	\overline{x_{7-i}(z)}x_7(z),
	\quad 1\leq i\leq6,
\end{equation}
and
\(
|x_7(z)|=1.
\)
Since $x_i\in H^2$, \eqref{eq:boundary-relation-degree} implies
\(
\overline{x_{7-i}}x_7\in H^2.
\)
Taking complex conjugates, we obtain
\(
x_{7-i}\overline{x_7}\in\overline{H^2}.
\)
Since $x_7$ is inner, multiplication by $x_7$ is unitary on
$L^2(\mathbb T)$. Hence
\[
x_{7-i}
=
x_7(x_{7-i}\overline{x_7})
\in x_7\overline{H^2}.
\]
Using $\overline{H^2}=zH^2_-$, it follows that
\(
x_{7-i}\in (zx_7)H^2_-.
\)
Consequently,
\begin{equation}\label{eq:model-intersection}
	x_{7-i}
	\in
	H^2\cap(zx_7)H^2_-,
	\quad 1\leq i\leq6.
\end{equation}
Since $zx_7$ is a finite Blaschke product, it is an inner function.
Moreover,
\[
\deg(zx_7)=n+1.
\]
For an inner function $\theta$, one has
\(
H^2\cap\theta H^2_-
=
H^2\ominus\theta H^2.
\)
Therefore, if
\[
\mathcal K_{\mathbf{x}}
:=
H^2\ominus(zx_7)H^2,
\]
then
\begin{equation}\label{eq:model-space-x}
	H^2\cap(zx_7)H^2_-
	=
	\mathcal K_{\mathbf{x}}.
\end{equation}
The space $\mathcal K_{\mathbf{x}}$ is the model space associated
with the finite Blaschke product $zx_7$; see, for example,
\cite[Definition 3.1.1]{Nikolski}. Since $zx_7$ has degree $n+1$,
\(
\dim\mathcal K_{\mathbf{x}}
=
n+1.
\)
This motivates the following definition.
\begin{defn}\label{Defn:degree-Gamma}
	Let
	\(
	\mathbf{x}=(x_1,\ldots,x_7)
	\)
	be a rational $\Gamma_{E(3;3;1,1,1)}$-inner function. We define its
	degree by
	\(
		\deg(\mathbf{x}):=\deg(x_7).
	\)
	Thus, if $\deg(x_7)=n$, then
	\[
	\dim\mathcal K_{\mathbf{x}}=n+1=\deg(\mathbf{x})+1.
	\]
\end{defn}
We next introduce the corresponding notion for rational
$\Gamma_{E(3;2;1,2)}$-inner functions.
Let
\(
\widetilde{\mathbf{x}}
=
(x_1,x_2,x_3,y_1,y_2)
\)
be a rational $\Gamma_{E(3;2;1,2)}$-inner function. By Lemma~\ref{Lem 3},
$x_3$ is an inner function. Since $x_3$ is rational, it is a finite
Blaschke product. Suppose that
\(
\deg(x_3)=n.
\)
As in the preceding argument, Lemma~\ref{Lem 3} gives
\[
x_1(z)=\overline{y_2(z)}x_3(z),
\quad
x_2(z)=\overline{y_1(z)}x_3(z)
\]
for almost every $z\in\mathbb T$. Hence
\(
y_2\overline{x_3},
y_1\overline{x_3}
\in H^2.
\)
Taking conjugates and using the fact that $x_3$ is inner, we obtain
\[
y_2,y_1
\in x_3\overline{H^2}
=
(zx_3)H^2_-.
\]
Since $y_1,y_2\in H^2$, it follows that
\[
y_j
\in
H^2\cap(zx_3)H^2_-,
\quad j=1,2.
\]
Thus, defining
\begin{equation}\label{eq:model-space-Theta}
	\mathcal K_{\widetilde{\mathbf{x}}}
	:=
	H^2\ominus(zx_3)H^2,
\end{equation}
we have
\[
H^2\cap(zx_3)H^2_-
=
\mathcal K_{\widetilde{\mathbf{x}}}.
\]
Since $zx_3$ is a finite Blaschke product of degree $n+1$,
\(
\dim\mathcal K_{\widetilde{\mathbf{x}}}
=
n+1.
\)
We therefore make the following definition.
\begin{defn}\label{Defn:degree-Theta}
	Let
	\(
	\widetilde{\mathbf{x}}
	=
	(x_1,x_2,x_3,y_1,y_2)
	\)
	be a rational $\Gamma_{E(3;2;1,2)}$-inner function. We define its
	degree by
	\(
		\deg(\widetilde{\mathbf{x}}):=\deg(x_3).
	\)
	Thus, if $\deg(x_3)=n$, then
	\[
	\dim\mathcal K_{\widetilde{\mathbf{x}}}
	=
	n+1
	=
	\deg(\widetilde{\mathbf{x}})+1.
	\]
\end{defn}
In what follows, we assume that the degrees of the rational inner
functions under consideration are positive.
We now introduce a topological notion of degree. Recall that the
degree of a continuous map from $\mathbb T$ into $\mathbb T$ can be
identified with the integer induced on the fundamental group. We
shall use the same idea for the  sets $K$ and
$K_1$.

\begin{defn}\label{Defn 3}
	Let
	\(
	\mathbf{x}=(x_1,\ldots,x_7)
	\)
	be a rational $\Gamma_{E(3;3;1,1,1)}$-inner function. Since
	$\mathbf{x}$ is continuous on $\mathbb T$ and
	\(
	\mathbf{x}(\mathbb T)\subseteq K,
	\)
	it induces a homomorphism
	\[
	\mathbf{x}_*:
	\pi_1(\mathbb T,1)
	\longrightarrow
	\pi_1(K,\mathbf{x}(1)).
	\]
	After identifying $\pi_1(K,\mathbf{x}(1))$ with $\mathbb Z$, we
	define the topological degree of $\mathbf{x}$ by
	\[
	\deg_{\mathrm{top}}(\mathbf{x})
	:=
	\mathbf{x}_*(1),
	\]
	where $1\in\mathbb Z=\pi_1(\mathbb T,1)$ denotes the standard
	generator.
\end{defn}
We next determine the fundamental group of $K$.
\begin{lem}\label{Lem 5}
	The space $K$ is homotopy equivalent to $\mathbb T$. In particular,
	\[
	\pi_1(K)\cong\mathbb Z.
	\]
\end{lem}

\begin{proof}
	Define maps
	\(
	f:K\longrightarrow\mathbb T\) by 
	\[f(x_1,\ldots,x_7)=x_7,
	\]
	and
	\(
	g:\mathbb T\longrightarrow K\) by \[
	g(\lambda)=(0,\ldots,0,\lambda).
	\]
	The map $g$ is well defined since
	\(
	(0,\ldots,0,\lambda)\in K
	~\text{for every }~\lambda\in\mathbb T.
	\)
Clearly,
	\(
	f\circ g=\operatorname{id}_{\mathbb T}.
	\)
On the other hand, for
$\mathbf{x}=(x_1,\ldots,x_7)\in K$ and $t\in[0,1]$, define
\[
H(\mathbf{x},t)
=
(tx_1,\ldots,tx_6,x_7).
\]
Since $\mathbf{x}\in K$, we have
\(
x_i=\overline{x_{7-i}}x_7, 1\leq i\leq6.
\)
Consequently,
\(
tx_i
=
t\overline{x_{7-i}}x_7
=
\overline{tx_{7-i}}x_7, 1\leq i\leq6,
\)
and hence
\(
H(\mathbf{x},t)\in K.
\)
Thus $H$ is a continuous map from $K\times[0,1]$ into $K$ satisfying
\[
H(\mathbf{x},0)
=
(0,\ldots,0,x_7)
=
(g\circ f)(\mathbf{x})
\]
and
\(
H(\mathbf{x},1)=\mathbf{x}.
\)
Therefore
\(
g\circ f\simeq\operatorname{id}_K.
\)
Together with
\(
f\circ g=\operatorname{id}_{\mathbb T},
\)
this shows that $f$ and $g$ are homotopy inverses. Hence
\(
K\simeq\mathbb T.
\)
Therefore
\(
\pi_1(K)\cong\pi_1(\mathbb T)\cong\mathbb Z.
\)
\end{proof}

\begin{prop}\label{Prop 3}
	Let
	\(
	\mathbf{x}=(x_1,\ldots,x_7)
	\)
	be a rational $\Gamma_{E(3;3;1,1,1)}$-inner function. Then its
	topological degree coincides with its degree in the sense of
	Definition~\ref{Defn:degree-Gamma}. More precisely,
	\[
		\deg_{\mathrm{top}}(\mathbf{x})
		=
		\deg(\mathbf{x})
		=
		\deg(x_7).
	\]
\end{prop}

\begin{proof}
	Consider the map
	\[
	f:K\longrightarrow\mathbb T,
	\qquad
	f(x_1,\ldots,x_7)=x_7.
	\]
	For the rational $\Gamma_{E(3;3;1,1,1)}$-inner function
	$\mathbf{x}$, we have
	\(
	f\circ\mathbf{x}=x_7
	\)
	as maps from $\mathbb T$ into $\mathbb T$.
Passing to the fundamental groups gives
	\(
	f_*\circ\mathbf{x}_*
	=
	(x_7)_*.
	\)
	By Lemma~\ref{Lem 5}, $f$ is a homotopy equivalence and hence
	\[
	f_*:\pi_1(K)\longrightarrow\pi_1(\mathbb T)
	\]
	is an isomorphism. Therefore, under the canonical identification
	of $\pi_1(K)$ with $\mathbb Z$ induced by $f$,
	\[
	\deg_{\mathrm{top}}(\mathbf{x})
	=
	(x_7)_*(1).
	\]
	Since $x_7$ is a finite Blaschke product,
	\(
	(x_7)_*(1)=\deg(x_7).
	\)
	Thus
	\(
	\deg_{\mathrm{top}}(\mathbf{x})
	=
	\deg(x_7)
	=
	\deg(\mathbf{x}).
	\)
\end{proof}
We now obtain the corresponding topological description for
$\Gamma_{E(3;2;1,2)}$.

\begin{lem}\label{lem:K1-homotopy}
	The set  $K_1$ of
	$\Gamma_{E(3;2;1,2)}$ is homotopy equivalent to $\mathbb T$.
	In particular,
	\[
	\pi_1(K_1)\cong\mathbb Z.
	\]
\end{lem}

\begin{proof}
	Define
	\(
	f_1:K_1\longrightarrow\mathbb T\) by
	\[f_1(x_1,x_2,x_3,y_1,y_2)=x_3,
	\]
	and
	\(
	g_1:\mathbb T\longrightarrow K_1\) by 
	\[
	g_1(\lambda)=(0,0,\lambda,0,0).
	\]
	By the characterization of $K_1$ in Proposition~\ref{Prop K_1},
	$g_1$ is well defined. Clearly,
	\(
	f_1\circ g_1=\operatorname{id}_{\mathbb T}.
	\)
For
\(
\widetilde{\mathbf{x}}
=(x_1,x_2,x_3,y_1,y_2)\in K_1
\)
and $t\in[0,1]$, define
\[
H_1(\widetilde{\mathbf{x}},t)
=
(tx_1,tx_2,x_3,ty_1,ty_2).
\]
Since
\(
x_1=\overline{y_2}x_3,
x_2=\overline{y_1}x_3
\)
on $K_1$, we have
\(
tx_1=\overline{ty_2}x_3,
tx_2=\overline{ty_1}x_3.
\)
Hence
\(
H_1(\widetilde{\mathbf{x}},t)\in K_1.
\)
Moreover,
\[
H_1(\widetilde{\mathbf{x}},0)
=
(0,0,x_3,0,0)
=
(g_1\circ f_1)(\widetilde{\mathbf{x}})
\]
and
\(
H_1(\widetilde{\mathbf{x}},1)
=
\widetilde{\mathbf{x}}.
\)
Thus
\(
g_1\circ f_1\simeq\operatorname{id}_{K_1}.
\)
Therefore $K_1$ is homotopy equivalent to $\mathbb T$, and hence
\(
\pi_1(K_1)\cong\mathbb Z.
\)
\end{proof}

\begin{defn}\label{defn:top-degree-Theta}
	Let
	\(
	\widetilde{\mathbf{x}}
	=
	(x_1,x_2,x_3,y_1,y_2)
	\)
	be a rational $\Gamma_{E(3;2;1,2)}$-inner function. Then
	$\widetilde{\mathbf{x}}$ induces a homomorphism
	\[
	\widetilde{\mathbf{x}}_*:
	\pi_1(\mathbb T,1)
	\longrightarrow
	\pi_1(K_1,\widetilde{\mathbf{x}}(1)).
	\]
	After identifying $\pi_1(K_1)$ with $\mathbb Z$ through the
	map induced by
	\[
	f_1(x_1,x_2,x_3,y_1,y_2)=x_3,
	\]
	we define
	\(
	\deg_{\mathrm{top}}(\widetilde{\mathbf{x}})
	:=
	\widetilde{\mathbf{x}}_*(1).
	\)
\end{defn}

\begin{prop}\label{prop:degree-Theta}
	Let
	\(
	\widetilde{\mathbf{x}}
	=
	(x_1,x_2,x_3,y_1,y_2)
	\)
	be a rational $\Gamma_{E(3;2;1,2)}$-inner function. Then
	\[
		\deg_{\mathrm{top}}(\widetilde{\mathbf{x}})
	=
	\deg(\widetilde{\mathbf{x}})
	=
	\deg(x_3).
	\]
\end{prop}

\begin{proof}
	By definition of $f_1$,
	\(
	f_1\circ\widetilde{\mathbf{x}}=x_3.
	\)
	Hence, on fundamental groups,
	\[
	(f_1)_*\circ\widetilde{\mathbf{x}}_*
	=
	(x_3)_*.
	\]
	By Lemma~\ref{lem:K1-homotopy}, $(f_1)_*$ is an isomorphism.
	Therefore, after identifying $\pi_1(K_1)$ with $\mathbb Z$,
	\[
	\deg_{\mathrm{top}}(\widetilde{\mathbf{x}})
	=
	(x_3)_*(1).
	\]
	Since $x_3$ is a finite Blaschke product,
	\(
	(x_3)_*(1)=\deg(x_3).
	\)
	Thus
	\(
	\deg_{\mathrm{top}}(\widetilde{\mathbf{x}})
	=
	\deg(x_3)
	=
	\deg(\widetilde{\mathbf{x}}).
	\)
\end{proof}

\begin{cor}\label{cor:model-topological-degree}
	Let $\mathbf{x}$ be a rational
	$\Gamma_{E(3;3;1,1,1)}$-inner function. Then
	\(
	\dim\mathcal K_{\mathbf{x}}
	=
	\deg_{\mathrm{top}}(\mathbf{x})+1.
	\)
	Similarly, if $\widetilde{\mathbf{x}}$ is a rational
	$\Gamma_{E(3;2;1,2)}$-inner function, then
	\(
	\dim\mathcal K_{\widetilde{\mathbf{x}}}
	=
	\deg_{\mathrm{top}}(\widetilde{\mathbf{x}})+1.
	\)
\end{cor}

\begin{proof}
	For $\mathbf{x}$, Definition~\ref{Defn:degree-Gamma} and
	Proposition~\ref{Prop 3} give
	\(
	\deg(\mathbf{x})
	=
	\deg_{\mathrm{top}}(\mathbf{x}).
	\)
	From the construction of $\mathcal K_{\mathbf{x}}$,
	\(
	\dim\mathcal K_{\mathbf{x}}
	=
	\deg(x_7)+1
	=
	\deg(\mathbf{x})+1.
	\)
	Hence
	\(
	\dim\mathcal K_{\mathbf{x}}
	=
	\deg_{\mathrm{top}}(\mathbf{x})+1.
	\)
	The proof for $\widetilde{\mathbf{x}}$ is identical, using
	\(
	\dim\mathcal K_{\widetilde{\mathbf{x}}}
	=
	\deg(x_3)+1.
	\)
\end{proof}

\begin{rem}\label{rem:two-degree-notions}
	The preceding results show that the two natural notions of degree
	considered in this section coincide. For a rational
	$\Gamma_{E(3;3;1,1,1)}$-inner function $\mathbf{x}$,
	\[
		\deg(\mathbf{x})
		=
		\deg(x_7)
		=
		\deg_{\mathrm{top}}(\mathbf{x}),
	\]
	while for a rational $\Gamma_{E(3;2;1,2)}$-inner function
	$\widetilde{\mathbf{x}}$,
	\[
		\deg(\widetilde{\mathbf{x}})
		=
		\deg(x_3)
		=
		\deg_{\mathrm{top}}(\widetilde{\mathbf{x}}).
	\]
	Thus the degree can equivalently be interpreted analytically,
	through the associated model space, and topologically, through the
	induced homomorphism on the fundamental group.
\end{rem}	

\section{Relations Among $\Gamma_{E(3;3;1,1,1)}$-Inner Functions, 
$\Gamma_{E(3;2;1,2)}$-Inner Functions, 
$\Gamma_{E(2;1;2)}$-Inner Functions and 
$\Gamma_{E(2;2;1,1)}$-Inner Functions}
\label{Relation}

In this section, we establish several relations among rational
$\Gamma_{E(3;3;1,1,1)}$-inner functions,
$\Gamma_{E(3;2;1,2)}$-inner functions,
$\Gamma_{E(2;1;2)}$-inner functions, and
$\Gamma_{E(2;2;1,1)}$-inner functions. We first construct several
families of rational $\Gamma_{E(3;3;1,1,1)}$-inner functions and
$\Gamma_{E(3;2;1,2)}$-inner functions from a given inner function.
We then investigate their relations with rational
$\Gamma_{E(2;2;1,1)}$-inner functions.
The notion of a $\Gamma_{E(2;2;1,1)}$-inner function was introduced
and studied in \cite{Alsalhi}. We begin with two immediate
consequences of Lemmas~\ref{Lem 1} and \ref{Lem 3}. We therefore
omit their proofs.

\begin{lem}\label{Lem 8}
Let
\(
\mathbf{x}=(x_1,\ldots,x_7)
\)
be a rational $\Gamma_{E(3;3;1,1,1)}$-inner function. Then, for each
$1\leq i\leq 6$,
\(
(x_i,x_{7-i},x_7)
\)
is a rational $\Gamma_{E(2;2;1,1)}$-inner function.
\end{lem}

\begin{lem}\label{Lem 10}
Let
\(
\widetilde{\mathbf{x}}
=(x_1,x_2,x_3,y_1,y_2)
\)
be a rational $\Gamma_{E(3;2;1,2)}$-inner function. Then
\(
(x_1,y_2,x_3)\) and \(
\left(\frac{x_2}{2},\frac{y_1}{2},x_3\right)
\)
are rational $\Gamma_{E(2;2;1,1)}$-inner functions.
\end{lem}
The following lemma provides an invariance property of rational
$\Gamma_{E(2;2;1,1)}$-inner functions.

\begin{lem}\label{Lem 7}
Let
\(
\mathbf{h}=(h_1,h_2,h_3)
\)
be a triple of rational functions defined on $\mathbb{D}$. Then
$\mathbf{h}$ is a rational $\Gamma_{E(2;2;1,1)}$-inner function if and
only if, for every $\omega\in\mathbb{T}$,
\(
(\omega h_1,\omega h_2,\omega^2h_3)
\)
is a rational $\Gamma_{E(2;2;1,1)}$-inner function.
\end{lem}

\begin{proof}
Suppose that $\mathbf{h}$ is a rational
$\Gamma_{E(2;2;1,1)}$-inner function. Then
\(
(h_1(z),h_2(z),h_3(z))
\in b\Gamma_{E(2;2;1,1)}
\)
for almost every $z\in\mathbb{T}$. By \cite[Lemma 2.1]{Pal}, the
distinguished boundary $b\Gamma_{E(2;2;1,1)}$ is invariant under the
transformation
\[
(a,b,c)\longmapsto
(\omega a,\omega b,\omega^2c),
\qquad \omega\in\mathbb{T}.
\]
Consequently,
\(
(\omega h_1(z),\omega h_2(z),\omega^2h_3(z))
\in b\Gamma_{E(2;2;1,1)}
\)
for every $\omega\in\mathbb{T}$ and almost every $z\in\mathbb{T}$.
Since the transformed triple is rational and maps $\mathbb{D}$ into
$\Gamma_{E(2;2;1,1)}$, it is a rational
$\Gamma_{E(2;2;1,1)}$-inner function.

Conversely, taking $\omega=1$ gives the desired conclusion.
\end{proof}
We next establish a function-theoretic analogue of
\cite[Lemma 3.1]{apal2} for rational
$\Gamma_{E(3;3;1,1,1)}$-inner functions.

\begin{prop}\label{Prop 5}
Let
\(
\mathbf{x}=(x_1,\ldots,x_7)
\)
be a $7$-tuple of rational functions defined on $\mathbb{D}$. Then
the following statements are equivalent:
\begin{enumerate}
\item[(1)] $\mathbf{x}$ is a rational
$\Gamma_{E(3;3;1,1,1)}$-inner function.

\item[(2)] For every $\omega\in\mathbb{T}$,
\(
(\omega x_1,x_2,\omega x_3,\omega x_4,
\omega^2x_5,\omega x_6,\omega^2x_7)
\)
is a rational $\Gamma_{E(3;3;1,1,1)}$-inner function.

\item[$(2')$] For every $\omega\in\mathbb{T}$,
\(
(\omega x_1,\omega x_2,\omega^2x_3,x_4,
\omega x_5,\omega x_6,\omega^2x_7)
\)
is a rational $\Gamma_{E(3;3;1,1,1)}$-inner function.

\item[$(2'')$] For every $\omega\in\mathbb{T}$,
\(
(x_1,\omega x_2,\omega x_3,\omega x_4,
\omega x_5,\omega^2x_6,\omega^2x_7)
\)
is a rational $\Gamma_{E(3;3;1,1,1)}$-inner function.
\end{enumerate}
\end{prop}
\begin{proof}
We prove $(1)\Longleftrightarrow(2)$. The proofs of
$(1)\Longleftrightarrow(2')$ and
$(1)\Longleftrightarrow(2'')$ are analogous.

Suppose that $\mathbf{x}$ is a rational
$\Gamma_{E(3;3;1,1,1)}$-inner function. Then
\(
\mathbf{x}(z)\in K
\)
for almost every $z\in\mathbb{T}$. By \cite[Lemma 3.1]{apal2},
\[
\begin{aligned}
(\omega x_1(z),x_2(z),\omega x_3(z),\omega x_4(z),\omega^2x_5(z),\omega x_6(z),\omega^2x_7(z))
\in K
\end{aligned}
\]
for every $\omega\in\mathbb{T}$ and almost every $z\in\mathbb{T}$.
Moreover, by \cite[Lemma 2.4]{apal2},
\[
\begin{aligned}
(\omega x_1(z),x_2(z),\omega x_3(z),\omega x_4(z),\omega^2x_5(z),\omega x_6(z),\omega^2x_7(z))
\in \Gamma_{E(3;3;1,1,1)}
\end{aligned}
\]
for every $z\in\overline{\mathbb{D}}$ and every
$\omega\in\mathbb{T}$. Hence
\(
(\omega x_1,x_2,\omega x_3,\omega x_4,
\omega^2x_5,\omega x_6,\omega^2x_7)
\)
is a rational $\Gamma_{E(3;3;1,1,1)}$-inner function.

The converse follows immediately by taking $\omega=1$.
\end{proof}
The corresponding invariance property for rational
$\Gamma_{E(3;2;1,2)}$-inner functions is given below.
\begin{prop}\label{Prop 6}
Let
\(
\widetilde{\mathbf{x}}
=(x_1,x_2,x_3,y_1,y_2)
\)
be a $5$-tuple of rational functions defined on $\mathbb{D}$. Then
$\widetilde{\mathbf{x}}$ is a rational
$\Gamma_{E(3;2;1,2)}$-inner function if and only if, for every
$\omega\in\mathbb{T}$,
\(
(x_1,\omega x_2,\omega^2x_3,\omega y_1,\omega^2y_2)
\)
is a rational $\Gamma_{E(3;2;1,2)}$-inner function.
\end{prop}

\begin{proof}
Suppose that $\widetilde{\mathbf{x}}$ is a rational
$\Gamma_{E(3;2;1,2)}$-inner function. Then
\[
\widetilde{\mathbf{x}}(z)
=(x_1(z),x_2(z),x_3(z),y_1(z),y_2(z))
\in K_1
\]
for almost every $z\in\mathbb{T}$. By
\cite[Proposition 3.6]{apal2},
\[
(x_1(z),\omega x_2(z),\omega^2x_3(z),
\omega y_1(z),\omega^2y_2(z))
\in K_1
\]
for every $\omega\in\mathbb{T}$ and almost every $z\in\mathbb{T}$.
On the other hand, by \cite[Lemma 2.8]{apal2},
\[
(x_1(z),\omega x_2(z),\omega^2x_3(z),
\omega y_1(z),\omega^2y_2(z))
\in \Gamma_{E(3;2;1,2)}
\]
for every $z\in\overline{\mathbb{D}}$ and every
$\omega\in\mathbb{T}$. Therefore,
\(
(x_1,\omega x_2,\omega^2x_3,\omega y_1,\omega^2y_2)
\)
is a rational $\Gamma_{E(3;2;1,2)}$-inner function.

The converse follows by taking $\omega=1$.
\end{proof}

For a $7$-tuple
$\mathbf{x}=(x_1,\ldots,x_7)$ of rational functions, define
\[
\begin{aligned}
\mathcal{A}_1(z_1)
&=
\left(
\frac{x_2-z_1x_3}{1-z_1x_1},
\frac{x_4-z_1x_5}{1-z_1x_1},
\frac{x_6-z_1x_7}{1-z_1x_1}
\right),\\
\mathcal{A}_2(z_2)
&=
\left(
\frac{x_1-z_2x_3}{1-z_2x_2},
\frac{x_4-z_2x_6}{1-z_2x_2},
\frac{x_5-z_2x_7}{1-z_2x_2}
\right),\\
\mathcal{A}_3(z_3)
&=
\left(
\frac{x_1-z_3x_5}{1-z_3x_4},
\frac{x_2-z_3x_6}{1-z_3x_4},
\frac{x_3-z_3x_7}{1-z_3x_4}
\right).
\end{aligned}
\]
The following theorem is a function-theoretic analogue of
\cite[Theorem 3.14]{apal1} and \cite[Theorem 3.2]{apal2}.

\begin{thm}\label{Thm 1}
Let
\(
\mathbf{x}=(x_1,\ldots,x_7)
\)
be a $7$-tuple of rational functions defined on $\mathbb{D}$. Then
the following statements are equivalent:

\begin{enumerate}
\item[(1)] $\mathbf{x}$ is a rational
$\Gamma_{E(3;3;1,1,1)}$-inner function.

\item[(2)] $\mathbf{x}(\mathbb{D})\subseteq
\Gamma_{E(3;3;1,1,1)}$ and $x_7$ is inner.

\item[(3)] $\mathbf{x}(\mathbb{D})\subseteq
\Gamma_{E(3;3;1,1,1)}$ and at least one of the following conditions
holds:
\begin{enumerate}
\item[(a)] If
\(
|x_2(z)|<1, |x_4(z)|<1
\)
for almost every $z\in\mathbb{T}$, then
$\mathcal{A}_2(z_2)$ and $\mathcal{A}_3(z_3)$ are rational
$\Gamma_{E(2;2;1,1)}$-inner functions for every
$z_2,z_3\in\mathbb{T}$.

\item[(b)] If
\(
|x_2(z)|=1
\)
for almost every $z\in\mathbb{T}$, then
$\mathcal{A}_2(z_2)$ is a rational
$\Gamma_{E(2;2;1,1)}$-inner function for every
$z_2\in\mathbb{D}$.

\item[(c)] If
\(
|x_4(z)|=1
\)
for almost every $z\in\mathbb{T}$, then
$\mathcal{A}_3(z_3)$ is a rational
$\Gamma_{E(2;2;1,1)}$-inner function for every
$z_3\in\mathbb{D}$.
\end{enumerate}

\item[(4)] $\mathbf{x}(\mathbb{D})\subseteq
\Gamma_{E(3;3;1,1,1)}$ and at least one of the following conditions
holds:
\begin{enumerate}
\item[(a)] If
\(
|x_1(z)|<1, |x_2(z)|<1
\)
for almost every $z\in\mathbb{T}$, then
$\mathcal{A}_1(z_1)$ and $\mathcal{A}_2(z_2)$ are rational
$\Gamma_{E(2;2;1,1)}$-inner functions for every
$z_1,z_2\in\mathbb{T}$.

\item[(b)] If
\(
|x_1(z)|=1
\)
for almost every $z\in\mathbb{T}$, then
$\mathcal{A}_1(z_1)$ is a rational
$\Gamma_{E(2;2;1,1)}$-inner function for every
$z_1\in\mathbb{D}$.

\item[(c)] If
\(
|x_2(z)|=1
\)
for almost every $z\in\mathbb{T}$, then
$\mathcal{A}_2(z_2)$ is a rational
$\Gamma_{E(2;2;1,1)}$-inner function for every
$z_2\in\mathbb{D}$.
\end{enumerate}

\item[(5)] $\mathbf{x}(\mathbb{D})\subseteq
\Gamma_{E(3;3;1,1,1)}$ and at least one of the following conditions
holds:
\begin{enumerate}
\item[(a)] If
\(
|x_1(z)|<1,|x_4(z)|<1
\)
for almost every $z\in\mathbb{T}$, then
$\mathcal{A}_1(z_1)$ and $\mathcal{A}_3(z_3)$ are rational
$\Gamma_{E(2;2;1,1)}$-inner functions for every
$z_1,z_3\in\mathbb{T}$.

\item[(b)] If
\(
|x_1(z)|=1
\)
for almost every $z\in\mathbb{T}$, then
$\mathcal{A}_1(z_1)$ is a rational
$\Gamma_{E(2;2;1,1)}$-inner function for every
$z_1\in\mathbb{D}$.

\item[(c)] If
\(
|x_4(z)|=1
\)
for almost every $z\in\mathbb{T}$, then
$\mathcal{A}_3(z_3)$ is a rational
$\Gamma_{E(2;2;1,1)}$-inner function for every
$z_3\in\mathbb{D}$.
\end{enumerate}
\end{enumerate}
\end{thm}

\begin{proof}
We prove
\(
(1)\Longleftrightarrow(2)\Longleftrightarrow(3).
\)
The proofs of
\(
(1)\Longleftrightarrow(2)\Longleftrightarrow(4)
\)
and
\(
(1)\Longleftrightarrow(2)\Longleftrightarrow(5)
\)
are analogous, after the appropriate permutation of the coordinates.

The implication $(1)\Rightarrow(2)$ follows immediately from the
definition of a rational $\Gamma_{E(3;3;1,1,1)}$-inner function.
Conversely, suppose that
\(
\mathbf{x}(\mathbb{D})\subseteq
\Gamma_{E(3;3;1,1,1)}
\)
and that $x_7$ is inner. Since
\(
|x_7(z)|=1
\)
for almost every $z\in\mathbb{T}$, Proposition~\ref{Prop K} implies
that
\(
\mathbf{x}(z)\in K
\)
for almost every $z\in\mathbb{T}$. Hence $\mathbf{x}$ is a rational
$\Gamma_{E(3;3;1,1,1)}$-inner function. Thus
\(
(1)\Longleftrightarrow(2).
\)

We next prove $(1)\Rightarrow(3)$. Let $\mathbf{x}$ be a rational
$\Gamma_{E(3;3;1,1,1)}$-inner function. By Lemma~\ref{Lem 1},
\(
x_i(z)=\overline{x_{7-i}(z)}x_7(z),1\leq i\leq6,
\)
and
\(
|x_7(z)|=1
\)
for almost every $z\in\mathbb{T}$. Hence
\(
\mathbf{x}(z)\in K
\)
for almost every $z\in\mathbb{T}$.
We consider the following three cases.

\medskip
\noindent
\textbf{Case 1: $(1)\Rightarrow(3)(a)$:}
Suppose that
\(
|x_2(z)|<1,~|x_4(z)|<1
\)
for almost every $z\in\mathbb{T}$, and let
$z_2,z_3\in\mathbb{T}$. Then
\[
1-z_2x_2(z)\neq0,
\quad
1-z_3x_4(z)\neq0
\]
for almost every $z\in\mathbb{T}$.
By Lemma~\ref{Lem 8}, the triple
$(x_2,x_5,x_7)$ is a rational
$\Gamma_{E(2;2;1,1)}$-inner function. Therefore,
\[
\left|
\frac{x_5(z)-z_2x_7(z)}
     {1-z_2x_2(z)}
\right|=1, \quad z_2\in \mathbb T
\]
for almost every $z\in\mathbb{T}$. Moreover, by
\cite[Theorem 2.22]{apal1},
\[
\left(
\frac{x_1-z_2x_3}{1-z_2x_2},
\frac{x_4-z_2x_6}{1-z_2x_2},
\frac{x_5-z_2x_7}{1-z_2x_2}
\right)
\in\Gamma_{E(2;2;1,1)} ,\quad z_2\in \mathbb T.
\]
It follows from \cite[Theorem 7.1]{Abouhajar} that
\[
\left(
\frac{x_1(z)-z_2x_3(z)}{1-z_2x_2(z)},
\frac{x_4(z)-z_2x_6(z)}{1-z_2x_2(z)},
\frac{x_5(z)-z_2x_7(z)}{1-z_2x_2(z)}
\right)
\in b\Gamma_{E(2;2;1,1)},\quad z_2\in \mathbb T
\]
for almost every $z\in\mathbb{T}$. Hence
$\mathcal{A}_2(z_2)$ is a rational
$\Gamma_{E(2;2;1,1)}$-inner function.
Similarly, by applying Lemma~\ref{Lem 8} to the triple
$(x_4,x_3,x_7)$, we obtain that
$\mathcal{A}_3(z_3)$ is a rational
$\Gamma_{E(2;2;1,1)}$-inner function for every
$z_3\in\mathbb{T}$.

\medskip
\noindent
\textbf{Case 2: $(1)\Rightarrow(3)(b)$:}
Suppose that
\(
|x_2(z)|=1
\)
for almost every $z\in\mathbb{T}$, and let $z_2\in\mathbb{D}$. Then
\(
|z_2x_2(z)|<1,
\)
and hence
\(
1-z_2x_2(z)\neq0
\)
for almost every $z\in\mathbb{T}$. By Lemma~\ref{Lem 8},
\[
\left|
\frac{x_5(z)-z_2x_7(z)}
     {1-z_2x_2(z)}
\right|=1, \quad z_2\in \mathbb D
\]
for almost every $z\in\mathbb{T}$. Furthermore,
\cite[Theorem 2.29]{apal1} gives
\[
\left(
\frac{x_1-z_2x_3}{1-z_2x_2},
\frac{x_4-z_2x_6}{1-z_2x_2},
\frac{x_5-z_2x_7}{1-z_2x_2}
\right)
\in\Gamma_{E(2;2;1,1)}, \quad z_2\in \mathbb D.
\]
Therefore, its boundary values belong to
$b\Gamma_{E(2;2;1,1)}$ almost everywhere on $\mathbb{T}$.
Consequently, $\mathcal{A}_2(z_2)$ is a rational
$\Gamma_{E(2;2;1,1)}$-inner function.

\medskip
\noindent
\textbf{Case 3: $(1)\Rightarrow(3)(c)$:}
The proof is analogous to that of Case 2, with $x_2$ and
$\mathcal{A}_2$ replaced by $x_4$ and $\mathcal{A}_3$, respectively.
Hence $\mathcal{A}_3(z_3)$ is a rational
$\Gamma_{E(2;2;1,1)}$-inner function for every $z_3\in\mathbb{D}$.
Thus,
\(
(1)\Longrightarrow(3).
\)

It remains to prove $(3)\Rightarrow(1)$.

\medskip
\noindent
\textbf{Case 1: $(3)(a)\Rightarrow(1)$:}
Suppose that
\(
|x_2(z)|<1, |x_4(z)|<1
\)
for almost every $z\in\mathbb{T}$, and that
$\mathcal{A}_2(z_2)$ and $\mathcal{A}_3(z_3)$ are rational
$\Gamma_{E(2;2;1,1)}$-inner functions for every
$z_2,z_3\in\mathbb{T}$. Then, for almost every $z\in\mathbb{T}$,
\[
\left(
\frac{x_1(z)-z_2x_3(z)}{1-z_2x_2(z)},
\frac{x_4(z)-z_2x_6(z)}{1-z_2x_2(z)},
\frac{x_5(z)-z_2x_7(z)}{1-z_2x_2(z)}
\right)
\in b\Gamma_{E(2;2;1,1)}
\]
and
\[
\left(
\frac{x_1(z)-z_3x_5(z)}{1-z_3x_4(z)},
\frac{x_2(z)-z_3x_6(z)}{1-z_3x_4(z)},
\frac{x_3(z)-z_3x_7(z)}{1-z_3x_4(z)}
\right)
\in b\Gamma_{E(2;2;1,1)}.
\]
Therefore, by \cite[Theorem 3.14]{apal1},
\(
\mathbf{x}(z)\in K
\)
for almost every $z\in\mathbb{T}$. Since
\(
\mathbf{x}(\mathbb{D})\subseteq
\Gamma_{E(3;3;1,1,1)},
\)
we conclude that $\mathbf{x}$ is a rational
$\Gamma_{E(3;3;1,1,1)}$-inner function.

\medskip
\noindent
\textbf{Case 2: $(3)(b)\Rightarrow(1)$.}
Suppose that
\(
|x_2(z)|=1
\)
for almost every $z\in\mathbb{T}$, and that
$\mathcal{A}_2(z_2)$ is a rational
$\Gamma_{E(2;2;1,1)}$-inner function for every $z_2\in\mathbb{D}$.
Then
\[
\left(
\frac{x_1(z)-z_2x_3(z)}{1-z_2x_2(z)},
\frac{x_4(z)-z_2x_6(z)}{1-z_2x_2(z)},
\frac{x_5(z)-z_2x_7(z)}{1-z_2x_2(z)}
\right)
\in b\Gamma_{E(2;2;1,1)}
\]
for almost every $z\in\mathbb{T}$. Hence, by
\cite[Theorem 3.14]{apal1},
\(
\mathbf{x}(z)\in K
\)
for almost every $z\in\mathbb{T}$. Consequently,
$\mathbf{x}$ is a rational $\Gamma_{E(3;3;1,1,1)}$-inner function.

\medskip
\noindent
\textbf{Case 3: $(3)(c)\Rightarrow(1)$.}
The proof is analogous to that of Case 5, with $x_2$ and
$\mathcal{A}_2$ replaced by $x_4$ and $\mathcal{A}_3$, respectively.
Thus,
\(
(1)\Longleftrightarrow(2)\Longleftrightarrow(3).
\)
This completes
the proof.
\end{proof}
Next, we introduce the following rational map:
\begin{equation*}
\mathcal{B}(w)
=
\left(
\frac{y_1-2wy_2}{2-y_1w},
\frac{2x_1-wx_2}{2-y_1w},
\frac{x_2-2wx_3}{2-y_1w}
\right). 
\end{equation*}
The following theorem provides an analogue of Theorem~\ref{Thm 1}
for rational $\Gamma_{E(3;2;1,2)}$-inner functions. Its proof follows
from the same arguments used in the proof of Theorem~\ref{Thm 1}, and
we therefore omit the details.

\begin{thm}\label{Thm 2}
Let
\(
\widetilde{\mathbf{x}}
=(x_1,x_2,x_3,y_1,y_2)
\)
be a $5$-tuple of rational functions defined on $\mathbb{D}$. Then
the following statements are equivalent:
\begin{enumerate}
\item[(1)] $\widetilde{\mathbf{x}}$ is a rational
$\Gamma_{E(3;2;1,2)}$-inner function.

\item[(2)] $\widetilde{\mathbf{x}}(\mathbb{D})\subseteq
\Gamma_{E(3;2;1,2)}$ and $x_3$ is inner.

\item[(3)] $\widetilde{\mathbf{x}}(\mathbb{D})\subseteq
\Gamma_{E(3;2;1,2)}$ and at least one of the following conditions
holds:
\begin{enumerate}
\item[(a)] If
\(
|y_1(z)|<2
\)
for almost every $z\in\mathbb{T}$, then, for every $w\in\mathbb{T}$,
$\mathcal{B}(w)$ is a rational
$\Gamma_{E(2;2;1,1)}$-inner function.

\item[(b)] If
\(
|y_1(z)|=2
\)
for almost every $z\in\mathbb{T}$, then, for every $w\in\mathbb{D}$,
$\mathcal{B}(w)$ is a rational
$\Gamma_{E(2;2;1,1)}$-inner function.
\end{enumerate}
\end{enumerate}
\end{thm}
Finally, we establish a relation between rational
$\Gamma_{E(3;2;1,2)}$-inner functions and rational
$\Gamma_{E(2;1;2)}$-inner functions. For this purpose, define
\begin{equation*}
\mathcal{C}(w)
=
\left(
\frac{y_1-wx_2}{1-wx_1},
\frac{y_2-wx_3}{1-wx_1}
\right).
\end{equation*}
The following theorem can be proved by applying arguments similar to
those used in the proof of Theorem~\ref{Thm 1}. We therefore omit the
proof.

\begin{thm}\label{Thm 3}
Let
\(
\widetilde{\mathbf{x}}
=(x_1,x_2,x_3,y_1,y_2)
\)
be a $5$-tuple of rational functions defined on $\mathbb{D}$. Then
the following statements are equivalent:
\begin{enumerate}
\item[(1)] $\widetilde{\mathbf{x}}$ is a rational
$\Gamma_{E(3;2;1,2)}$-inner function.

\item[(2)] $\widetilde{\mathbf{x}}(\mathbb{D})\subseteq
\Gamma_{E(3;2;1,2)}$ and $x_3$ is inner.

\item[(3)] $\widetilde{\mathbf{x}}(\mathbb{D})\subseteq
\Gamma_{E(3;2;1,2)}$ and at least one of the following conditions
holds:
\begin{enumerate}
\item[(a)] If
\(
|x_1(z)|<1
~\text{and}~
|x_3(z)|=1
\)
for almost every $z\in\mathbb{T}$, then, for every $w\in\mathbb{T}$,
$\mathcal{C}(w)$ is a rational
$\Gamma_{E(2;1;2)}$-inner function.

\item[(b)] If
\(
|x_1(z)|=1
\)
for almost every $z\in\mathbb{T}$, then, for every $w\in\mathbb{D}$,
$\mathcal{C}(w)$ is a rational
$\Gamma_{E(2;1;2)}$-inner function.
\end{enumerate}
\end{enumerate}
\end{thm}
We next establish two results describing the relationship between
rational $\Gamma_{E(3;3;1,1,1)}$-inner functions and rational
$\Gamma_{E(3;2;1,2)}$-inner functions. These results provide
substantial families of rational $\Gamma_{E(3;3;1,1,1)}$-inner
functions and rational $\Gamma_{E(3;2;1,2)}$-inner functions.

\begin{thm}\label{Thm 4}
Let
\(
\mathbf{x}=(x_1,\ldots,x_7)
\)
be a $7$-tuple of rational functions defined on $\mathbb{D}$. Then
the following statements are equivalent:
\begin{enumerate}
\item[(1)] $\mathbf{x}$ is a rational
$\Gamma_{E(3;3;1,1,1)}$-inner function.

\item[(2)] For every $\eta\in\mathbb{T}$,
\(
(x_1,x_3+\eta x_5,\eta x_7,
x_2+\eta x_4,\eta x_6)
\)
is a rational $\Gamma_{E(3;2;1,2)}$-inner function.

\item[(3)] For every $\eta\in\mathbb{T}$,
\(
(x_1,x_5+\eta x_3,\eta x_7,
x_4+\eta x_2,\eta x_6)
\)
is a rational $\Gamma_{E(3;2;1,2)}$-inner function.

\item[(4)] For every $\eta\in\mathbb{T}$,
\(
(x_1,x_3+\eta x_6,\eta x_7,
x_1+\eta x_4,\eta x_5)
\)
is a rational $\Gamma_{E(3;2;1,2)}$-inner function.

\item[(5)] For every $\eta\in\mathbb{T}$,
\(
(x_1,x_6+\eta x_3,\eta x_7,
x_4+\eta x_1,\eta x_5)
\)
is a rational $\Gamma_{E(3;2;1,2)}$-inner function.

\item[(6)] For every $\eta\in\mathbb{T}$,
\(
(x_1,x_5+\eta x_6,\eta x_7,
x_1+\eta x_2,\eta x_3)
\)
is a rational $\Gamma_{E(3;2;1,2)}$-inner function.

\item[(7)] For every $\eta\in\mathbb{T}$,
\(
(x_1,x_6+\eta x_5,\eta x_7,
x_2+\eta x_1,\eta x_3)
\)
is a rational $\Gamma_{E(3;2;1,2)}$-inner function.
\end{enumerate}
\end{thm}

\begin{proof}
We prove $(1)\Longleftrightarrow(2)$. The equivalences
\[
(1)\Longleftrightarrow(3),
(1)\Longleftrightarrow(4),
(1)\Longleftrightarrow(5),
(1)\Longleftrightarrow(6),~\text{and}~
(1)\Longleftrightarrow(7)
\]
follow by analogous arguments.

\medskip
\noindent
$(1)\Rightarrow(2)$:
Suppose that $\mathbf{x}$ is a rational
$\Gamma_{E(3;3;1,1,1)}$-inner function. Then
\(
\mathbf{x}(z)
=(x_1(z),\ldots,x_7(z))\in K
\)
for almost every $z\in\mathbb{T}$. By
\cite[Theorem 3.20]{apal1}, for every fixed
$\eta\in\mathbb{T}$,
\[
(x_1(z),x_3(z)+\eta x_5(z),\eta x_7(z),
x_2(z)+\eta x_4(z),\eta x_6(z))
\in K_1
\]
for almost every $z\in\mathbb{T}$.
Moreover, since
\(
\mathbf{x}(z)\in
\Gamma_{E(3;3;1,1,1)},
\)
\cite[Theorem 2.48]{apal1} gives
\[
(x_1(z),x_3(z)+\eta x_5(z),\eta x_7(z),
x_2(z)+\eta x_4(z),\eta x_6(z))
\in
\Gamma_{E(3;2;1,2)}
\]
for every $z\in\mathbb{D}$ and every $\eta\in\mathbb{T}$.
Therefore,
\(
(x_1,x_3+\eta x_5,\eta x_7,
x_2+\eta x_4,\eta x_6)
\)
is a $5$-tuple of rational functions mapping $\mathbb{D}$ into
$\Gamma_{E(3;2;1,2)}$ and mapping almost every point of
$\mathbb{T}$ into $K_1$. Hence it is a rational
$\Gamma_{E(3;2;1,2)}$-inner function for every
$\eta\in\mathbb{T}$.

\medskip
\noindent
$(2)\Rightarrow(1)$:
Suppose that, for every $\eta\in\mathbb{T}$,
\(
(x_1,x_3+\eta x_5,\eta x_7,
x_2+\eta x_4,\eta x_6)
\)
is a rational $\Gamma_{E(3;2;1,2)}$-inner function. By
Lemma~\ref{Lem 3}, for almost every $z\in\mathbb{T}$,
\begin{equation}\label{T 7}
x_1(z)
=
\overline{\eta x_6(z)}\,\eta x_7(z)
=
\overline{x_6(z)}x_7(z),
\quad \eta\in\mathbb{T},
\end{equation}
and
\begin{equation}\label{T 8}
x_3(z)+\eta x_5(z)
=
\overline{x_2(z)+\eta x_4(z)}\,\eta x_7(z),
\qquad \eta\in\mathbb{T}.
\end{equation}
Expanding the right-hand side of \eqref{T 8}, we obtain
\[
x_3(z)+\eta x_5(z)
=
\overline{x_2(z)}\,\eta x_7(z)
+
\overline{x_4(z)}\,x_7(z).
\]
Hence
\begin{equation}\label{T 9}
x_3(z)-\overline{x_4(z)}x_7(z)
+
\eta\bigl(
x_5(z)-\overline{x_2(z)}x_7(z)
\bigr)
=0.
\end{equation}
Since \eqref{T 9} holds for every $\eta\in\mathbb{T}$, both the
constant term and the coefficient of $\eta$ must vanish. Thus,
\(
x_3(z)=\overline{x_4(z)}x_7(z)
\)
and
\(
x_5(z)=\overline{x_2(z)}x_7(z),
\)
or, equivalently,
\(
x_2(z)=\overline{x_5(z)}x_7(z),
\)
for almost every $z\in\mathbb{T}$. Together with \eqref{T 7}, we
obtain
\[
x_1(z)=\overline{x_6(z)}x_7(z),
\quad
x_2(z)=\overline{x_5(z)}x_7(z),
\quad
x_3(z)=\overline{x_4(z)}x_7(z).
\]
Furthermore, since $|\eta x_7(z)|=1$ for almost every
$z\in\mathbb{T}$, we have
\(
|x_7(z)|=1
\)
for almost every $z\in\mathbb{T}. $ Therefore,
\(
\mathbf{x}(z)\in K
\)
for almost every $z\in\mathbb{T}.$
On the other hand, by \cite[Theorem 2.48]{apal1},
the corresponding $7$-tuple maps $\mathbb{D}$ into
$\Gamma_{E(3;3;1,1,1)}$. Consequently,
$\mathbf{x}$ is a rational
$\Gamma_{E(3;3;1,1,1)}$-inner function.
This proves $(1)\Longleftrightarrow(2)$, and the proof of the theorem
is complete.
\end{proof}

\begin{thm}\label{Thm 5}
Let
\(
\widetilde{\mathbf{x}}
=(x_1,x_2,x_3,y_1,y_2)
\)
be a $5$-tuple of rational functions defined on $\mathbb{D}$. Then
the following statements are equivalent:
\begin{enumerate}
\item[(1)] $\widetilde{\mathbf{x}}$ is a rational
$\Gamma_{E(3;2;1,2)}$-inner function.

\item[(2)] The $7$-tuple
\(
\left(
x_1,\frac{y_1}{2},\frac{x_2}{2},
\frac{y_1}{2},\frac{x_2}{2},y_2,x_3
\right)
\)
is a rational $\Gamma_{E(3;3;1,1,1)}$-inner function.
\end{enumerate}
\end{thm}

\begin{proof}
Suppose first that $\widetilde{\mathbf{x}}$ is a rational
$\Gamma_{E(3;2;1,2)}$-inner function. By Lemma~\ref{Lem 3}, for
almost every $z\in\mathbb{T}$,
\[
x_1(z)=\overline{y_2(z)}x_3(z),
\quad
x_2(z)=\overline{y_1(z)}x_3(z),
\quad
|x_3(z)|=1.
\]
Consequently,
\[
x_1(z)=\overline{y_2(z)}x_3(z),\quad
\frac{y_1(z)}{2}
=
\frac{\overline{x_2(z)}}{2}x_3(z),\quad\text{
and}\quad
\frac{x_2(z)}{2}
=
\frac{\overline{y_1(z)}}{2}x_3(z)
\]
for almost every $z\in\mathbb{T}$. Together with
\(
|x_3(z)|=1,
\)
these are precisely the boundary relations required for the
$7$-tuple
\(
\left(
x_1,\frac{y_1}{2},\frac{x_2}{2},
\frac{y_1}{2},\frac{x_2}{2},y_2,x_3
\right)
\)
to belong to 
$K$ of $\Gamma_{E(3;3;1,1,1)}$.
Moreover,
\(
\widetilde{\mathbf{x}}(\mathbb{D})
\subseteq
\Gamma_{E(3;2;1,2)}.
\)
Hence, by \cite[Theorem 3.5]{Bharali},
\[
\left(
x_1(z),\frac{y_1(z)}{2},\frac{x_2(z)}{2},
\frac{y_1(z)}{2},\frac{x_2(z)}{2},
y_2(z),x_3(z)
\right)
\in
\Gamma_{E(3;3;1,1,1)}
\]
for every $z\in\mathbb{D}$. Therefore,
\(
\left(
x_1,\frac{y_1}{2},\frac{x_2}{2},
\frac{y_1}{2},\frac{x_2}{2},y_2,x_3
\right)
\)
is a rational $\Gamma_{E(3;3;1,1,1)}$-inner function.

Conversely, suppose that
\(
\left(
x_1,\frac{y_1}{2},\frac{x_2}{2},
\frac{y_1}{2},\frac{x_2}{2},y_2,x_3
\right)
\)
is a rational $\Gamma_{E(3;3;1,1,1)}$-inner function. The converse
implication then follows from Theorem~\ref{Thm 4}, applied to the
appropriate choice of coordinates. Hence
$\widetilde{\mathbf{x}}$ is a rational
$\Gamma_{E(3;2;1,2)}$-inner function.
\end{proof}
	
\section{Examples of $\Gamma_{E(3;3;1,1,1)}$- and
$\Gamma_{E(3;2;1,2)}$-Inner Functions}
\label{Example}

The definitions of $\Gamma_{E(3;3;1,1,1)}$-inner and
$\Gamma_{E(3;2;1,2)}$-inner functions naturally suggest methods for
constructing concrete examples of such functions. In particular, once
the  $K$ and $K_1$ are properly understood,
the construction of inner functions becomes relatively transparent.
In this section, we present several explicit examples of
$\Gamma_{E(3;3,1,1,1)}$- and $\Gamma_{E(3;2;1,2)}$-inner functions.
We begin with two non-constant scalar inner functions.

\begin{exam}\label{Exam 1}
Let $\varphi,\psi\in H^\infty(\mathbb D)$ be two non-constant inner
functions, and let
\[
\mathbb B_3
=
\{A\in M_3(\mathbb C):\|A\|<1\}
\]
denote the open unit ball of $M_3(\mathbb C)$.

\noindent\textbf{(A)}
Consider the diagonal matrix-valued function
\[
h(z)
=
\begin{pmatrix}
\varphi(z) & 0 & 0\\
0 & \psi(z) & 0\\
0 & 0 & \varphi(z)\psi(z)
\end{pmatrix},
\quad z\in\mathbb D.
\]
Since $\varphi$ and $\psi$ are non-constant inner functions, we have
\(
|\varphi(z)|<1, |\psi(z)|<1, z\in\mathbb D.
\)
Consequently,
\[
\|h(z)\|
=
\max\big\{
|\varphi(z)|,\,
|\psi(z)|,\,
|\varphi(z)\psi(z)|
\big\}
<1,
\quad z\in\mathbb D.
\]
Thus, $h:\mathbb D\to\mathbb B_3$ is holomorphic. Define
\(
\mathbf{x}
=
\Pi_{E(3;3;1,1,1)}\circ h.
\)
By \cite[Theorem 2.41]{apal1}, we obtain
\[
\begin{aligned}
\mathbf{x}(z)
={}
\big(
\varphi(z),\,
\psi(z),\,
\varphi(z)\psi(z),\,
\varphi(z)\psi(z),\,
\varphi^2(z)\psi(z),
\varphi(z)\psi^2(z),\,
\varphi^2(z)\psi^2(z)
\big)
\in
\Gamma_{E(3;3;1,1,1)},
\quad z\in\mathbb D.
\end{aligned}
\]
Hence, $\mathbf{x}$ is a holomorphic map from $\mathbb D$ into
$\Gamma_{E(3;3;1,1,1)}$. Since $\varphi$ and $\psi$ are inner
functions, their radial boundary values are unimodular almost
everywhere on $\mathbb T$. Therefore,
\[
|x_7(\zeta)|
=
|\varphi(\zeta)|^2|\psi(\zeta)|^2
=1
\]
for almost every $\zeta\in\mathbb T$. Moreover, for almost every
$\zeta\in\mathbb T$,
\[
x_i(\zeta)
=
\overline{x_{7-i}(\zeta)}\,x_7(\zeta),
\quad 1\leq i\leq6.
\]
Thus $\mathbf{x}(\zeta)\in K$ for almost every $\zeta\in\mathbb T$.
Hence $\mathbf{x}$ is a $\Gamma_{E(3;3;1,1,1)}$-inner function.

\medskip

\noindent\textbf{(B)}
The preceding example can be generalized by applying unitary
transformations to the matrix-valued function $h$. Consider the
unitary matrices
\[
U=
\begin{pmatrix}
1 & 0 & 0\\
0 & \dfrac{1}{\sqrt2} & \dfrac{1}{\sqrt2}\\
0 & -\dfrac{1}{\sqrt2} & \dfrac{1}{\sqrt2}
\end{pmatrix},
\quad
V=
\begin{pmatrix}
1&0&0\\
0&1&0\\
0&0&1
\end{pmatrix}.
\]
Define
\[
g(z)=Uh(z)V,\quad z\in\mathbb D.
\]
Since $U$ and $V$ are unitary, we have
\[
\|g(z)\|=\|h(z)\|<1,
\quad z\in\mathbb D.
\]
Therefore, $g:\mathbb D\to\mathbb B_3$ is holomorphic. Explicitly,
\[
g(z)
=
\begin{pmatrix}
\varphi(z) & 0 & 0\\[1mm]
0 & \dfrac{\psi(z)}{\sqrt2}
  & \dfrac{\varphi(z)\psi(z)}{\sqrt2}\\[2mm]
0 & -\dfrac{\psi(z)}{\sqrt2}
  & \dfrac{\varphi(z)\psi(z)}{\sqrt2}
\end{pmatrix}.
\]
Let
\(
\mathbf{x}
=
\Pi_{E(3;3;1,1,1)}\circ g.
\)
Then, by \cite[Theorem 2.41]{apal1},
\[
\begin{aligned}
\mathbf{x}(z)
=
\Big(
\varphi(z),\,
\dfrac{\psi(z)}{\sqrt2},\,
\dfrac{\varphi(z)\psi(z)}{\sqrt2},\,
\dfrac{\varphi(z)\psi(z)}{\sqrt2},\dfrac{\varphi^2(z)\psi(z)}{\sqrt2},\,
\varphi(z)\psi^2(z),\,
\varphi^2(z)\psi^2(z)
\Big),
\quad z\in\mathbb D.
\end{aligned}
\]
In particular,
\(
\mathbf{x}(\mathbb D)
\subseteq
\Gamma_{E(3;3;1,1,1)}.
\)
Furthermore, since $\varphi$ and $\psi$ are inner functions, the
radial boundary values of $\mathbf{x}$ satisfy
\(
|x_7(\zeta)|=1
\)
for almost every $\zeta\in\mathbb T$, and the defining relations of
 $K$ are satisfied almost everywhere.
Consequently,
\[
\mathbf{x}(\zeta)\in K
\]
for almost every $\zeta\in\mathbb T$. Hence $\mathbf{x}$ is a
$\Gamma_{E(3;3;1,1,1)}$-inner function.

\medskip
In Example~\ref{Exam 1}(A), the coordinates satisfy the relations
\begin{equation}\label{E 1}
\begin{aligned}
x_1x_2 =x_3,
x_1x_4 =x_5,
x_2x_4 =x_6,
x_1x_6 =x_7,
x_2x_5= x_7,
x_3x_4 =x_7,
\end{aligned}
\quad z\in\mathbb D.
\end{equation}
These identities illustrate the multiplicative structure underlying
this particular example.

\end{exam}
The preceding construction also provides examples of
$\Gamma_{E(3;2;1,2)}$-inner functions. Indeed, Theorem~\ref{Thm 4}
establishes a direct relationship between
$\Gamma_{E(3;3;1,1,1)}$-inner functions and
$\Gamma_{E(3;2;1,2)}$-inner functions. We therefore obtain the
following examples.

\begin{exam}\label{Exam 2}
Let $\varphi,\psi$ and $\mathbb B_3$ be as in
Example~\ref{Exam 1}.

\medskip

\noindent\textbf{(A)}
From Example~\ref{Exam 1}(A), the function
\(
\mathbf{x}
=
\big(
\varphi,\,
\psi,\,
\varphi\psi,\,
\varphi\psi,\,
\varphi^2\psi,\,
\varphi\psi^2,\,
\varphi^2\psi^2
\big)
\)
is a $\Gamma_{E(3;3;1,1,1)}$-inner function. Hence, by
Theorem~\ref{Thm 4}, for every fixed $\eta\in\mathbb T$, the function
\[
\widetilde{\mathbf{x}}_\eta
=
\big(
x_1,\,
x_3+\eta x_5,\,
\eta x_7,\,
x_2+\eta x_4,\,
\eta x_6
\big)
\]
is a $\Gamma_{E(3;2;1,2)}$-inner function. In the present case, this
gives
\[
\widetilde{\mathbf{x}}_\eta
=
\big(
\varphi,\,
\varphi\psi+\eta\varphi^2\psi,\,
\eta\varphi^2\psi^2,\,
\psi+\eta\varphi\psi,\,
\eta\varphi\psi^2
\big).
\]
Thus, for every $\eta\in\mathbb T$,
\(
\widetilde{\mathbf{x}}_\eta
=
\big(
\varphi,\,
\varphi\psi(1+\eta\varphi),\,
\eta\varphi^2\psi^2,\,
\psi(1+\eta\varphi),\,
\eta\varphi\psi^2
\big)
\)
is a $\Gamma_{E(3;2;1,2)}$-inner function.

\medskip

\noindent\textbf{(B)}
Similarly, starting from the
$\Gamma_{E(3;3;1,1,1)}$-inner function constructed in
Example~\ref{Exam 1}(B), Theorem~\ref{Thm 4} yields further examples
of $\Gamma_{E(3;2;1,2)}$-inner functions. For instance, using the
transformation
\[
\big(
x_1,\,
x_3+\eta x_5,\,
\eta x_7,\,
x_2+\eta x_4,\,
\eta x_6
\big),
\quad \eta\in\mathbb T,
\]
we obtain
\[
\begin{aligned}
\widetilde{\mathbf{x}}_\eta
=
\Big(
\varphi,\,
\dfrac{\varphi\psi}{\sqrt2}
+\eta\dfrac{\varphi^2\psi}{\sqrt2},\,
\eta\varphi^2\psi^2,\psi/\sqrt2
+\eta\varphi\psi/\sqrt2,\,
\eta\varphi\psi^2
\Big).
\end{aligned}
\]
Hence $\widetilde{\mathbf{x}}_\eta$ is a
$\Gamma_{E(3;2;1,2)}$-inner function for every $\eta\in\mathbb T$.

\end{exam}
	
\section{Descriptions of Rational $\Gamma_{E(3;3;1,1,1)}$-Inner
Functions and Rational $\Gamma_{E(3;2;1,2)}$-Inner Functions}
\label{Description}
It is well known that every rational inner function on $\mathbb D$
admits a representation of the form given in
\eqref{Description of inner function}. In this section, we develop
partially analogous representations for rational
$\Gamma_{E(3;3;1,1,1)}$-inner functions and rational
$\Gamma_{E(3;2;1,2)}$-inner functions. More precisely, we establish
necessary conditions that describe the component functions of such
rational inner functions in terms of polynomials. We then obtain
partial converses of these representations. These results constitute
two of the main results of the paper.
We begin with the case of rational
$\Gamma_{E(3;3;1,1,1)}$-inner functions.
\begin{thm}\label{Description 1}
Let
\(
\mathbf{x}=(x_1,\ldots,x_7)
\)
be a rational $\Gamma_{E(3;3;1,1,1)}$-inner function of degree $n$.
Then there exist polynomials
\(
E_1,\ldots,E_6,D
\)
such that the following conditions hold:
\begin{enumerate}
    \item
    \(
    \deg E_i\leq n,1\leq i\leq6,
    \deg D\leq n.
    \)

    \item
    \(
    D(z)\neq0, z\in\overline{\mathbb D}.
    \)

    \item
    \(
    x_i=\frac{E_i}{D}, 1\leq i\leq6,
    x_7=\frac{D^{\sim n}}{D}
    \)
    on $\overline{\mathbb D}$.

    \item
    \(
    |E_i(z)|\leq |D(z)|,
     z\in\overline{\mathbb D},
    1\leq i\leq6.
    \)

    \item
    \(
    E_i=E_{7-i}^{\sim n},
    1\leq i\leq6.
    \)
\end{enumerate}
Moreover, if
\(
\{E'_1,\ldots,E'_6,D'\}
\)
is another collection of polynomials satisfying $(1)$--$(5)$ and
representing the same function $\mathbf{x}$, then there exists
$t\in\mathbb R\setminus\{0\}$ such that
\[
D'=tD,
\quad
E'_i=tE_i, \quad1\leq i\leq6.
\]
Conversely, every such real non-zero scalar multiple also satisfies
$(1)$--$(5)$ and represents the same function $\mathbf{x}$.
\end{thm}

\begin{proof}
Suppose that
\(
\mathbf{x}=(x_1,\ldots,x_7)
\)
is a rational $\Gamma_{E(3;3;1,1,1)}$-inner function of degree $n$.
By Lemma~\ref{Lem 1}, $x_7$ is a rational inner function of degree
$n$. Therefore, by the representation
\eqref{Description of inner function}, there exists a polynomial $D$
of degree at most $n$ such that
\(
D(z)\neq0, z\in\overline{\mathbb D},
\)
and
\begin{equation}\label{x_7}
x_7=\frac{D^{\sim n}}{D}.
\end{equation}
Since $\deg(x_7)=n$, Proposition~\ref{Prop 1}(2) implies that every
pole of $x_i$, $1\leq i\leq6$, is also a pole of $x_7$, with
multiplicity no greater than the corresponding multiplicity of the
pole of $x_7$. Hence, for each $1\leq i\leq6$, there exists a
polynomial $E_i$ with $\deg E_i\leq n$ such that
\(
x_i=\frac{E_i}{D}.
\)
Since
\(
\mathbf{x}(\mathbb D)
\subseteq
\Gamma_{E(3;3;1,1,1)},
\)
we have
\[
|x_i(z)|\leq1,
\quad z\in\mathbb D,
\quad 1\leq i\leq6.
\]
By continuity, the same inequality holds on
$\overline{\mathbb D}$. Consequently,
\[
|E_i(z)|\leq |D(z)|,
\quad z\in\overline{\mathbb D},
\quad 1\leq i\leq6.
\]
This proves $(1)$--$(4)$.

It remains to establish $(5)$. By Lemma~\ref{Lem 2}, for almost every
$z\in\mathbb T$,
\begin{equation}\label{Des 1}
x_i(z)
=
x_{7-i}^{\vee}(1/z)x_7(z),
\quad 1\leq i\leq6.
\end{equation}
Using \eqref{x_7}, we obtain
\begin{equation}\label{Des 2}
\frac{E_i(z)}{D(z)}
=
\frac{\overline{E_{7-i}(1/\overline z)}}
     {\overline{D(1/\overline z)}}
\frac{z^n\overline{D(1/\overline z)}}{D(z)}
\end{equation}
for almost every $z\in\mathbb T$. Hence
\[
E_i(z)
=
z^n\overline{E_{7-i}(1/\overline z)}
=
E_{7-i}^{\sim n}(z),
\quad z\in\mathbb T.
\]
Since both sides are polynomials, the identity theorem yields
\(
E_i=E_{7-i}^{\sim n},
 1\leq i\leq6.
\)
Thus $(5)$ holds.
We now prove the uniqueness assertion. Let
\(
\{E'_1,\ldots,E'_6,D'\}
\)
be another collection of polynomials satisfying $(1)$--$(5)$ and
representing the same function $\mathbf{x}$. Then
\begin{equation}\label{Des 4}
\frac{E_i}{D}
=
\frac{E'_i}{D'},
\quad 1\leq i\leq6,
\quad\text{and}\qquad
\frac{D^{\sim n}}{D}
=
\frac{D'^{\sim n}}{D'}.
\end{equation}
Since $D$ and $D'$ have no zeros in $\overline{\mathbb D}$ and
represent the same rational inner function $x_7$, they have the same
zeros in $\mathbb C\setminus\overline{\mathbb D}$, counted with
multiplicity. Hence there exists $t\in\mathbb C\setminus\{0\}$ such
that
\(
D'=tD.
\)
Substituting this into the second equality in \eqref{Des 4}, we get
\(
\frac{D^{\sim n}}{D}
=
\frac{\overline{t}\,D^{\sim n}}{tD}.
\)
Therefore $\overline t=t$, and hence
\(
t\in\mathbb R\setminus\{0\}.
\)
Finally, the first equality in \eqref{Des 4} gives
\(
\frac{E_i}{D}
=
\frac{E'_i}{tD},
\)
and consequently
\(
E'_i=tE_i,
 1\leq i\leq6.
\)
The converse is immediate: if
\[
D'=tD,\quad E'_i=tE_i,
\quad t\in\mathbb R\setminus\{0\},
\]
then $(1)$--$(5)$ remain valid and the resulting rational functions
are unchanged. This completes the proof.
\end{proof}

\begin{rem}\label{Rem 2}
Let
\(
\mathbf{x}=(x_1,\ldots,x_7)
\)
be a rational $\Gamma_{E(3;3;1,1,1)}$-inner function. By
Lemma~\ref{Lem 1}, for almost every $z\in\mathbb T$,
\[
x_i(z)=\overline{x_{7-i}(z)}x_7(z),
\quad 1\leq i\leq6,\quad \text{
and}\quad
|x_7(z)|=1.
\]
Consequently,
\(
|x_i(z)|=|x_{7-i}(z)|, 1\leq i\leq6,
\)
for almost every $z\in\mathbb T$.
\end{rem}
Theorem~\ref{Description 1} provides an explicit polynomial
representation of every component of a rational
$\Gamma_{E(3;3;1,1,1)}$-inner function. However, the conditions
$(1)$--$(5)$ in that theorem are, in general, only necessary
conditions; they do not by themselves guarantee that
\(
\mathbf{x}=(x_1,\ldots,x_7)
\)
is a rational $\Gamma_{E(3;3;1,1,1)}$-inner function. We next obtain
a partial converse.

Suppose that there exist polynomials
\(
E_1,\ldots,E_6,D
\)
satisfying conditions $(1)$, $(2)$, $(4)$, and $(5)$ of
Theorem~\ref{Description 1}, and define
\begin{equation}\label{construction Gamma 3}
x_i=\frac{E_i}{D},
\quad 1\leq i\leq6,
\quad
x_7=\frac{D^{\sim n}}{D}.
\end{equation}
By condition $(5)$,
\(
E_i=E_{7-i}^{\sim n},
 1\leq i\leq6.
\)
Consequently, the boundary relations
\[
|x_7(z)|=1,
\quad
x_i(z)=\overline{x_{7-i}(z)}x_7(z),
\quad 1\leq i\leq6,
\]
hold for almost every $z\in\mathbb T$. Thus, the boundary relations
necessary for membership in $K$ are
automatically satisfied.
It remains to verify that $\mathbf{x}$ maps $\mathbb D$ into
$\Gamma_{E(3;3;1,1,1)}$. By \cite[Theorem 2.39]{apal1}, it is
sufficient to verify one of the following conditions:
\begin{enumerate}
    \item
    \(
    (x_2(z),x_4(z),x_6(z))
    \in\Gamma_{E(2;2;1,1)}
    \)
    and
    \(
    \|\Psi^{(1)}(\cdot,\mathbf{x}(z))\|_{H^\infty(\mathbb D^2)}
    \leq1;
    \)

    \item
    \(
    (x_1(z),x_4(z),x_5(z))
    \in\Gamma_{E(2;2;1,1)}
    \)
    and
    \(
    \|\Psi^{(2)}(\cdot,\mathbf{x}(z))\|_{H^\infty(\mathbb D^2)}
    \leq1;
    \)

    \item
    \(
    (x_1(z),x_2(z),x_3(z))
    \in\Gamma_{E(2;2;1,1)}
    \)
    and
    \(
    \|\Psi^{(3)}(\cdot,\mathbf{x}(z))\|_{H^\infty(\mathbb D^2)}
    \leq1.
    \)
\end{enumerate}
We consider the first condition. For
$z_2,z_3\in\mathbb T$, if
\(
1-z_2x_2(z)-z_3x_4(z)+z_2z_3x_6(z)\neq0,
\)
then Corollary~\ref{Cor 1} gives
\[
\left|
\frac{
x_1(z)-z_2x_3(z)-z_3x_5(z)+z_2z_3x_7(z)}
{
1-z_2x_2(z)-z_3x_4(z)+z_2z_3x_6(z)
}
\right|
\leq1,
\]
for $z_2,z_3\in\overline{\mathbb D}$. Hence a necessary condition for
$\mathbf{x}$ to be a $\Gamma_{E(3;3;1,1,1)}$-inner function is
\begin{equation}\label{Necessary 1}
1-z_2x_2(z)-z_3x_4(z)+z_2z_3x_6(z)\neq0,
\quad z_2,z_3\in\mathbb T.
\end{equation}
By \cite[Definition 1.1]{Abouhajar}, condition \eqref{Necessary 1}
is equivalent to
\begin{equation}\label{Necessary 2}
(x_2(z),x_4(z),x_6(z))
\in\Gamma_{E(2;2;1,1)},
\quad z\in\mathbb T.
\end{equation}

Conversely, if \eqref{Necessary 2} holds, then the defining
characterization of $\Gamma_{E(2;2;1,1)}$ implies that the denominator
of $\Psi^{(1)}(\cdot,\mathbf{x}(z))$ has no zeros in
$\overline{\mathbb D}^2$. Together with
\[
|x_7(z)|=1,
\quad
x_i(z)=\overline{x_{7-i}(z)}x_7(z),
\quad 1\leq i\leq6,
\]
this yields
\(
\|\Psi^{(1)}(\cdot,\mathbf{x}(z))\|_{H^\infty(\mathbb D^2)}
\leq1.
\)
Similarly, we can verify $(2)$ and $(3)$.
The following theorem gives a partial converse to
Theorem~\ref{Description 1}.
\begin{thm}\label{Partial Converse 1}
Let $E_1,\ldots,E_6,D$ be polynomials satisfying conditions
$(1)$, $(2)$, $(4)$, and $(5)$ of Theorem~\ref{Description 1}.
Define
\begin{equation}\label{x_i, x_7}
x_i=\frac{E_i}{D},
\quad 1\leq i\leq6,
\quad
x_7=\frac{D^{\sim n}}{D}.
\end{equation}
Suppose that at least one of the following conditions holds:
\begin{enumerate}
    \item
    \(
    (x_2(z),x_4(z),x_6(z))
    \in\Gamma_{E(2;2;1,1)},
     z\in\mathbb T;
    \)

    \item
    \(
    (x_1(z),x_4(z),x_5(z))
    \in\Gamma_{E(2;2;1,1)},
    z\in\mathbb T;
    \)

    \item
    \(
    (x_1(z),x_2(z),x_3(z))
    \in\Gamma_{E(2;2;1,1)},
     z\in\mathbb T.
    \)
\end{enumerate}
Then
\(
\mathbf{x}=(x_1,\ldots,x_7)
\)
is a rational $\Gamma_{E(3;3;1,1,1)}$-inner function of degree at
most $n$.
\end{thm}

\begin{proof}
Without loss of generality, assume that
\(
(x_2(z),x_4(z),x_6(z))
\in\Gamma_{E(2;2;1,1)},
 z\in\mathbb T.
\)
We divide the proof into three steps.

\medskip

\noindent\textbf{Step 1: $x_7$ is rational inner.}
Since $D$ has no zeros in $\overline{\mathbb D}$, the function
\[
x_7=\frac{D^{\sim n}}{D}
\]
is a rational inner function. If $D$ and $D^{\sim n}$ have common
zeros on $\mathbb T$, the corresponding common factors may cancel.
Therefore,
\(
\deg(x_7)\leq n.
\)

\medskip

\noindent\textbf{Step 2: $\mathbf{x}$ has boundary values in $K$.}
By condition $(5)$ of Theorem~\ref{Description 1},
\(
E_i=E_{7-i}^{\sim n},
1\leq i\leq6.
\)
Consequently, by the same calculation as in the proof of
Theorem~\ref{Description 1},
\(
x_i(z)
=
\overline{x_{7-i}(z)}x_7(z),
1\leq i\leq6,
\)
for almost every $z\in\mathbb T$. Since $|x_7(z)|=1$ almost
everywhere on $\mathbb T$, it follows from
\cite[Theorem 3.14]{apal1} that
\(
\mathbf{x}(z)\in K
\)
for almost every $z\in\mathbb T.$

\medskip

\noindent\textbf{Step 3: $\mathbf{x}$ maps $\mathbb D$ into
$\Gamma_{E(3;3;1,1,1)}$.}
By \cite[Theorem 2.39]{apal1}, it suffices to prove
\begin{equation}\label{Con Des 2}
\|\Psi^{(1)}(\cdot,\mathbf{x}(z))\|_{H^\infty(\mathbb D^2)}
\leq1,
\quad z\in\mathbb D.
\end{equation}
Since
\(
\mathbf{x}(\zeta)\in K
\subseteq
\Gamma_{E(3;3;1,1,1)}
\)
for almost every $\zeta\in\mathbb T$, 
\cite[ Theorem~2.39]{apal1} gives
\[
|\Psi^{(1)}(z_2,z_3,\mathbf{x}(\zeta))|
\leq1
\]
for every fixed $z_2,z_3\in\overline{\mathbb D}$ and for almost every
$\zeta\in\mathbb T$.
Fix $z_2,z_3\in\overline{\mathbb D}$. Since
\(
(x_2(\zeta),x_4(\zeta),x_6(\zeta))
\in\Gamma_{E(2;2;1,1)},
 \zeta\in\mathbb T,
\)
the denominator of
\(
\Psi^{(1)}(z_2,z_3,\mathbf{x}(\zeta))
\)
does not vanish for $\zeta\in\mathbb T$. Since each $x_i$ is
holomorphic on $\mathbb D$, the function
\[
z\longmapsto
\Psi^{(1)}(z_2,z_3,\mathbf{x}(z))
\]
is holomorphic on $\mathbb D$. The maximum modulus principle
therefore yields
\[
|\Psi^{(1)}(z_2,z_3,\mathbf{x}(z))|
\leq1,
\quad z\in\mathbb D.
\]
Since $z_2,z_3\in\overline{\mathbb D}$ were arbitrary,
\eqref{Con Des 2} follows. Hence
\(
\mathbf{x}(\mathbb D)
\subseteq
\Gamma_{E(3;3;1,1,1)}.
\)
Thus $\mathbf{x}$ is a rational map from $\mathbb D$ into
$\Gamma_{E(3;3;1,1,1)}$ whose radial boundary values belong to $K$
almost everywhere on $\mathbb T$. Therefore $\mathbf{x}$ is a
rational $\Gamma_{E(3;3;1,1,1)}$-inner function. Its degree is
$\deg(x_7)\leq n$. This completes the proof.
\end{proof}

\begin{rem}\label{Rem 3}
Theorem~\ref{Partial Converse 1} can likewise be proved by assuming
that either
\(
(x_1(z),x_4(z),x_5(z))
\in\Gamma_{E(2;2;1,1)}
\)
or
\(
(x_1(z),x_2(z),x_3(z))
\in\Gamma_{E(2;2;1,1)}
\)
for every $z\in\mathbb T$. The proof is analogous to the one given
above.
\end{rem}

We next establish the corresponding description for rational
$\Gamma_{E(3;2;1,2)}$-inner functions.

\begin{thm}\label{Description 2}
Let
\(
\widetilde{\mathbf{x}}
=
(x_1,x_2,x_3,y_1,y_2)
\)
be a rational $\Gamma_{E(3;2;1,2)}$-inner function of degree $n$.
Then there exist polynomials
\(
E_1,E_2,F_1,F_2,D
\)
such that the following conditions hold:
\begin{enumerate}
    \item
    \(
    \deg E_i\leq n,
    \deg F_i\leq n,
     1\leq i\leq2,
    \deg D\leq n.
    \)

    \item
    \(
    D(z)\neq0,
     z\in\overline{\mathbb D}.
    \)

    \item
    \(
    x_i=\frac{E_i}{D},
    y_i=\frac{F_i}{D}, 1\leq i\leq2,
    x_3=\frac{D^{\sim n}}{D}
    \)
    on $\overline{\mathbb D}$.

    \item
    \(
    |E_1(z)|\leq|D(z)|,
    |E_2(z)|\leq2|D(z)|,
    |F_1(z)|\leq2|D(z)|,
    |F_2(z)|\leq|D(z)|,
    \)
    for all $z\in\overline{\mathbb D}$.

    \item
    \(
    E_i=F_{3-i}^{\sim n},
    1\leq i\leq2.
    \)
\end{enumerate}
Moreover, if
\(
\{E'_1,E'_2,F'_1,F'_2,D'\}
\)
is another collection of polynomials satisfying $(1)$--$(5)$ and
representing the same function
$\widetilde{\mathbf{x}}$, then there exists
$t\in\mathbb R\setminus\{0\}$ such that
\[
D'=tD,
\quad
E'_i=tE_i,
\quad
F'_i=tF_i,
\quad 1\leq i\leq2.
\]
Conversely, every such real non-zero scalar multiple also satisfies
$(1)$--$(5)$ and represents the same function
$\widetilde{\mathbf{x}}$.
\end{thm}

\begin{proof}
The proof is analogous to that of Theorem~\ref{Description 1}. By
Lemma~\ref{Lem 3}, $x_3$ is a rational inner function of degree $n$.
Hence, by \eqref{Description of inner function},
\(
x_3=\frac{D^{\sim n}}{D},
\)
where $D$ has degree at most $n$ and has no zeros in
$\overline{\mathbb D}$.
By the pole relations for the components of a rational
$\Gamma_{E(3;2;1,2)}$-inner function, each of $x_1,x_2,y_1,y_2$
has poles contained among those of $x_3$, with no larger
multiplicities. Consequently, there exist polynomials
$E_1,E_2,F_1,F_2$ of degree at most $n$ such that
\[
x_i=\frac{E_i}{D},
\quad
y_i=\frac{F_i}{D},
\quad 1\leq i\leq2.
\]
The bounds in $(4)$ follow from
\[
|x_1|\leq1,\quad
|x_2|\leq2,\quad
|y_1|\leq2,\quad
|y_2|\leq1
\]
on $\overline{\mathbb D}$.
Finally, Lemma~\ref{Lem 3} gives, almost everywhere on $\mathbb T$,
\[
x_1=\overline{y_2}\,x_3,
\quad
x_2=\overline{y_1}\,x_3,
\quad
|x_3|=1.
\]
Substituting the polynomial representations and using the
representation of $x_3$ gives
\[
E_1=F_2^{\sim n},
\qquad
E_2=F_1^{\sim n}.
\]
Thus $(5)$ holds.
The uniqueness assertion follows exactly as in the proof of
Theorem~\ref{Description 1}. Indeed, two denominators representing
the same rational inner function $x_3$ differ by a non-zero scalar.
The self-inversive relation in $(5)$ forces this scalar to be real,
and the representations of $x_1,x_2,y_1,y_2$ then imply the
corresponding scalar relations for $E_i$ and $F_i$.
\end{proof}

\begin{rem}\label{Rem 4}
Let
\(
\widetilde{\mathbf{x}}
=
(x_1,x_2,x_3,y_1,y_2)
\)
be a rational $\Gamma_{E(3;2;1,2)}$-inner function. By
Lemma~\ref{Lem 3}, for almost every $z\in\mathbb T$,
\[
x_1(z)=\overline{y_2(z)}x_3(z),
\quad
x_2(z)=\overline{y_1(z)}x_3(z),
\quad
|x_3(z)|=1.
\]
Consequently,
\(
|x_1(z)|=|y_2(z)|,
|x_2(z)|=|y_1(z)|
\)
for almost every $z\in\mathbb T$.
\end{rem}
Theorem~\ref{Description 2} gives an explicit polynomial
representation of each component of a rational
$\Gamma_{E(3;2;1,2)}$-inner function. As in the preceding case, the
conditions $(1)$--$(5)$ are necessary but are not, in general,
sufficient to guarantee that
\[
\widetilde{\mathbf{x}}
=
(x_1,x_2,x_3,y_1,y_2)
\]
is a rational $\Gamma_{E(3;2;1,2)}$-inner function. We therefore
proceed to obtain a partial converse.
Let
\(
E_1,E_2,F_1,F_2,D
\)
be polynomials satisfying conditions $(1)$, $(2)$, $(4)$, and $(5)$
of Theorem~\ref{Description 2}, and define
\begin{equation}\label{construction Gamma 32}
x_i=\frac{E_i}{D},
\quad
y_i=\frac{F_i}{D},
\quad 1\leq i\leq2,
\quad
x_3=\frac{D^{\sim n}}{D}.
\end{equation}
By condition $(5)$,
\(
E_1=F_2^{\sim n},
E_2=F_1^{\sim n}.
\)
Consequently,
\[
x_1(z)=\overline{y_2(z)}x_3(z),
\quad
x_2(z)=\overline{y_1(z)}x_3(z),
\quad
|x_3(z)|=1
\]
for almost every $z\in\mathbb T$. Hence
\(
\widetilde{\mathbf{x}}(z)\in K_1
\)
for almost every $z\in\mathbb T$.
It remains to verify that
\[
\widetilde{\mathbf{x}}:\mathbb D
\longrightarrow
\Gamma_{E(3;2;1,2)}.
\]
By \cite[Theorem 3.5]{Bharali} and
\cite[Proposition 2.46]{apal1}, it is sufficient to verify that
\(
(y_1(z),y_2(z))
\in
\Gamma_{E(2;1;2)}
\)
and
\(
\|\Psi_3(\cdot,\widetilde{\mathbf{x}}(z))\|_{H^\infty(\mathbb D)}
\leq1,
 z\in\mathbb D.
\)
If
\(
1-wy_1(z)+w^2y_2(z)\neq0,
 w\in\mathbb T,
\)
then Corollary~\ref{Cor 1} yields
\[
\left|
\frac{x_3(z)w^2-x_2(z)w+x_1(z)}
{y_2(z)w^2-y_1(z)w+1}
\right|
\leq1,
\quad w\in\overline{\mathbb D}.
\]
Therefore, a necessary condition for
$\widetilde{\mathbf{x}}$ to be a
$\Gamma_{E(3;2;1,2)}$-inner function is
\begin{equation}\label{Necessary 3}
y_2(z)w^2-y_1(z)w+1\neq0,\quad
w\in\mathbb T.
\end{equation}
By the characterization of
$\Gamma_{E(2;1;2)}$, condition \eqref{Necessary 3} is equivalent to
\begin{equation}\label{Necessary 3'}
(y_1(z),y_2(z))
\in
\Gamma_{E(2;1;2)},
\quad z\in\mathbb T.
\end{equation}

Conversely, if \eqref{Necessary 3'} holds, then the defining
characterization of $\Gamma_{E(2;1;2)}$ implies that the denominator
of $\Psi_3$ has no zeros in $\mathbb D$. Together with
\[
x_1(z)=\overline{y_2(z)}x_3(z),
\quad
x_2(z)=\overline{y_1(z)}x_3(z),
\quad
|x_3(z)|=1,
\]
this yields
\(
\|\Psi_3(\cdot,\widetilde{\mathbf{x}}(z))\|_{H^\infty(\mathbb D)}
\leq1.
\)
This leads to the following partial converse.

\begin{thm}\label{Partial Converse 2}
Let $E_1,E_2,F_1,F_2,D$ be polynomials satisfying conditions
$(1)$, $(2)$, $(4)$, and $(5)$ of Theorem~\ref{Description 2}.
Define
\begin{equation}\label{x_i, y_i, x_3}
x_i=\frac{E_i}{D},
\quad
y_i=\frac{F_i}{D},
\quad 1\leq i\leq2,
\quad
x_3=\frac{D^{\sim n}}{D}.
\end{equation}
If
\(
(y_1(z),y_2(z))
\in
\Gamma_{E(2;1;2)}, z\in\mathbb T,
\)
then
\(
\widetilde{\mathbf{x}}
=
(x_1,x_2,x_3,y_1,y_2)
\)
is a rational $\Gamma_{E(3;2;1,2)}$-inner function of degree at most
$n$.
\end{thm}	
	
\section{Interpolation by Rational $\Gamma_{E(3;3;1,1,1)}$-Inner and
Rational $\Gamma_{E(3;2;1,2)}$-Inner Functions}\label{Interpolation}

In this section, we investigate interpolation problems for the domains
$G_{E(3;3,1,1,1)}$ and $G_{E(3;2;1,2)}$ and establish the existence of
rational $\Gamma_{E(3;3,1,1,1)}$-inner and
$\Gamma_{E(3;2;1,2)}$-inner interpolating functions. We first recall
some consequences of the classical Nevanlinna--Pick interpolation
theorem and its matrix-valued counterpart.

\begin{thm}[{\cite{McCarthy 1}}]\label{Interpolation in M_n(C)}
Let $z_1,\ldots,z_m$ be distinct points in $\mathbb D$, and let
$W_1,\ldots,W_m\in M_n(\mathbb C)$. Then there exists a matrix-valued
Schur function
\(
F:\mathbb D\longrightarrow M_n(\mathbb C)
\)
such that
\[
F(z_j)=W_j,\quad 1\leq j\leq m,
\]
if and only if the block Pick matrix
\[
\left(
\frac{I_n-W_iW_j^*}{1-z_i\overline{z}_j}
\right)_{i,j=1}^{m}
\]
is positive semidefinite.
\end{thm}
We shall also use the following interpolation result.
\begin{thm}[{\cite[Theorem 1.1]{costara}}]\label{Inter}
Let $W_1,\ldots,W_m\in M_n(\mathbb C)$ satisfy
\[
r(W_j)<1,\quad 1\leq j\leq m,
\]
where $r(\cdot)$ denotes the spectral radius. Then there exists a
holomorphic function
\(
G:\mathbb D\longrightarrow M_n(\mathbb C)
\)
such that
\(
\sup_{z\in\mathbb D}\|G(z)\|<1
\)
and
\(
G(z_j)=X_jW_jX_j^{-1},
1\leq j\leq m,
\)
for some invertible matrices $X_1,\ldots,X_m\in M_n(\mathbb C)$.
\end{thm}
We begin with a scalar interpolation consequence of the
Schwarz-type functions associated with
$\Gamma_{E(3;3,1,1,1)}$.

\begin{thm}\label{Interpolation 1}
Let
\(
h=(x_1,\ldots,x_7):\mathbb D\longrightarrow
G_{E(3;3,1,1,1)}
\)
be holomorphic, and let
\[
h(z_j)
=
(x_{1,j},x_{2,j},x_{3,j},x_{4,j},x_{5,j},x_{6,j},x_{7,j}),
\quad 1\leq j\leq N,
\]
where $z_1,\ldots,z_N$ are distinct points in $\mathbb D$.
For fixed $a,b\in\mathbb D$, define
\[
\Psi_{a,b}^{(1)}(\mathbf{x})
=
\frac{
x_1-a x_2-bx_5+abx_7
}{
1-a x_3-bx_4+abx_6
}.
\]
Then
\(
\Psi_{a,b}^{(1)}\circ h:\mathbb D\longrightarrow\mathbb D.
\)
Consequently, for every $a,b\in\mathbb D$,
\[
\left(
\frac{
1-\overline{\Psi_{a,b}^{(1)}(h(z_j))}
\Psi_{a,b}^{(1)}(h(z_i))
}{
1-\overline{z_j}z_i
}
\right)_{i,j=1}^{N}
\geq 0.
\]
Similarly, if
\[
\Psi_{a,b}^{(2)}(\mathbf{x})
=
\frac{
x_2-a x_1-bx_6+abx_7
}{
1-a x_3-bx_4+abx_5
}\quad \text{
and}\quad
\Psi_{a,b}^{(3)}(\mathbf{x})
=
\frac{
x_4-a x_5-bx_2+abx_7
}{
1-a x_1-bx_3+abx_6
},
\]
then
\(
\Psi_{a,b}^{(2)}\circ h,\,
\Psi_{a,b}^{(3)}\circ h
\in H^\infty(\mathbb D),
\)
with norm at most one, and the corresponding Pick matrices are
positive semidefinite.
\end{thm}

\begin{proof}
Fix $a,b\in\mathbb D$. By the Schwarz-type characterization of
$G_{E(3;3,1,1,1)}$ (see \cite[Theorem 2.22]{apal1}), the function
\(
\Psi_{a,b}^{(1)}
\)
maps $G_{E(3;3,1,1,1)}$ into $\mathbb D$. Hence
\(
F_{a,b}(z)
=
\Psi_{a,b}^{(1)}(h(z))
\)
belongs to the scalar Schur class. Since
\(
F_{a,b}(z_j)
=
\Psi_{a,b}^{(1)}(h(z_j)),
\)
the classical Nevanlinna--Pick theorem yields
\[
\left(
\frac{
1-\overline{F_{a,b}(z_j)}F_{a,b}(z_i)
}{
1-\overline{z_j}z_i
}
\right)_{i,j=1}^{N}
\geq0.
\]
Substituting the definition of $F_{a,b}$ gives the asserted
positivity.
The assertions concerning $\Psi_{a,b}^{(2)}$ and
$\Psi_{a,b}^{(3)}$ follow in exactly the same manner.
\end{proof}
We next formulate the interpolation theorem for
$G_{E(3;3,1,1,1)}$ in terms of matrix-valued Schur functions. 
We shall use the fact from \cite[Theorem 2.24]{apal1} that 
\[
\Pi_{E(3;3,1,1,1)}
\bigl(B_{M_3(\mathbb C)}(0,1)\bigr)
\subseteq
G_{E(3;3,1,1,1)}.
\]
For $z=(z_1,z_2,z_3)\in\mathbb D^3$, put
\(
\pi_{E(3;3,1,1,1)}(z_1,z_2,z_3)
=
\Pi_{E(3;3,1,1,1)}
\left(\begin{smallmatrix}
z_1&0&0\\
0&z_2&0\\
0&0&z_3
\end{smallmatrix}\right).
\)

\begin{thm}\label{Interpolation in G_{E(3;3,1,1,1)}}
Let
\(
z_1,\ldots,z_N\in\mathbb D
\)
be distinct points, and let
\[
\mathbf{x}_j=(x_{1j},\dots,x_{7j})
\in
\Pi_{E(3;3,1,1,1)}
\bigl(B_{M_3(\mathbb C)}(0,1)\bigr),\quad
1\leq j\leq N.
\]
Suppose that there exists a holomorphic function
\[
h:\mathbb D\longrightarrow
\Pi_{E(3;3,1,1,1)}
\bigl(B_{M_3(\mathbb C)}(0,1)\bigr)
\]
such that
\(
h(z_j)=\mathbf{x}_j,
1\leq j\leq N.
\)
Then there exists a rational
$\Gamma_{E(3;3,1,1,1)}$-inner function
\[
\widetilde h:\mathbb D\longrightarrow
\Gamma_{E(3;3,1,1,1)}
\]
such that
\(
\widetilde h(z_j)=\mathbf{x}_j,
1\leq j\leq N.
\) 

Moreover, the rational $\Gamma_{E(3;3;1,1,1)}$-inner function
$\widetilde h$ admits the following explicit representation. There
exist $n\in\mathbb N$, $\xi\in\mathbb T$, and
$\alpha_1,\ldots,\alpha_n\in\mathbb D$ such that, for each
$1\leq k\leq3$, there exist
\(
t_k,q_k\in\mathbb N\cup\{0\},
r_k\in\mathbb R,
\)
and
\(
\eta_k,\delta_k,\xi_{k,1},\ldots,\xi_{k,t_k}\in\mathbb T,
\beta_{k,1},\ldots,\beta_{k,q_k}\in\mathbb D,
\)
satisfying
\begin{equation}\label{*}
t_k+q_k=n,
\quad
\eta_k\delta_k
\prod_{i=1}^{t_k}\xi_{k,i}
=\xi.
\end{equation}
Then
\(
\widetilde h=(x_1,\ldots,x_7),
\)
where
\begin{equation}\label{x_7 Representation}
x_7(z)
=
\xi\prod_{i=1}^n
\frac{z-\alpha_i}
     {1-\overline{\alpha_i}z},
\end{equation}
and, for $1\leq k\leq3$,
\begin{equation}\label{x_k_x_7-k_Representation}
\begin{aligned}
x_k(z)
=
r_k\eta_k
\frac{
\displaystyle
\left(\prod_{i=1}^{t_k}(\xi_{k,i}+z)\right)
\left(\prod_{i=1}^{q_k}f_{k,i}(z)\right)
}{
\displaystyle
\prod_{i=1}^{n}(\overline{\alpha_i}z-1)
},\quad
x_{7-k}(z)
=
r_k\delta_k
\frac{
\displaystyle
\left(\prod_{i=1}^{t_k}(\xi_{k,i}+z)\right)
\left(\prod_{i=1}^{q_k}g_{k,i}(z)\right)
}{
\displaystyle
\prod_{i=1}^{n}(\overline{\alpha_i}z-1)
},
\end{aligned}
\end{equation}
where, for each $1\leq i\leq q_k$, either
\(
f_{k,i}(z)=\overline{\beta}_{k,i}z-1,
g_{k,i}(z)=\beta_{k,i}-z,
\)
or
\(
f_{k,i}(z)=\beta_{k,i}-z,
g_{k,i}(z)=\overline{\beta}_{k,i}z-1.
\)
In particular, for $z\in\mathbb T$, the factors in
\eqref{x_k_x_7-k_Representation}, together with condition \eqref{*},
give
\(
x_k(z)
=
\overline{x_{7-k}(z)}\,x_7(z),
 1\leq k\leq3.
\)
Moreover,
\(
|x_7(z)|=1, z\in\mathbb T.
\)
Consequently,
\(
\widetilde h(\mathbb T)\subseteq K.
\)
Thus $\widetilde h$ is a rational
$\Gamma_{E(3;3;1,1,1)}$-inner function whose boundary values lie in
$K$.

\end{thm}

\begin{proof}
Let $M$ be a holomorphic lifting associated with
$\Pi_{E(3;3,1,1,1)}$, so that
\[
\Pi_{E(3;3,1,1,1)}\circ M=\operatorname{id}
\]
on the relevant domain and
\(
M\circ h:\mathbb D\longrightarrow B_{M_3(\mathbb C)}(0,1).
\)
Put
\(
f=M\circ h.
\)
Then
\(
\|f(z)\|<1, z\in\mathbb D.
\)
In particular,
\(
r(f(z))<1, z\in\mathbb D.
\)
For $0<\epsilon<1$, define
\(
f_\epsilon(z)=\epsilon f(z).
\)
Then
\(
r(f_\epsilon(z))<1,z\in\mathbb D.
\)
By Theorem~\ref{Inter}, there exists a holomorphic function
\[
F_\epsilon:\mathbb D\longrightarrow M_3(\mathbb C)
\]
such that
\(
\sup_{z\in\mathbb D}\|F_\epsilon(z)\|<1
\)
and
\(
F_\epsilon(z_j)
\sim
\epsilon f(z_j),
1\leq j\leq N.
\)
Here $\sim$ denotes similarity.
Since the family $\{F_\epsilon:0<\epsilon<1\}$ is locally bounded,
Montel's theorem yields a sequence
$\epsilon_\nu\to1$ such that
\(
F_{\epsilon_\nu}\longrightarrow F
\)
locally uniformly on $\mathbb D$, where
\(
F:\mathbb D\longrightarrow
\overline{B_{M_3(\mathbb C)}(0,1)}.
\)
Moreover, the similarity relation at the interpolation nodes implies
\[
\sigma(F(z_j))
=
\sigma(f(z_j)),
\quad
1\leq j\leq N.
\]
Since $\Pi_{E(3;3,1,1,1)}$ depends only on the spectral data,
we obtain
\[
\Pi_{E(3;3,1,1,1)}(F(z_j))
=
\Pi_{E(3;3,1,1,1)}(f(z_j))
=
h(z_j).
\]
By the matrix-valued Nevanlinna--Pick interpolation theorem, there
exists a rational matrix inner function
\(
F_0:\mathbb D\longrightarrow M_3(\mathbb C)
\)
which interpolates the same spectral data at the points $z_j$.
Thus,
\[
F_0(z_j)\sim f(z_j),
\quad
1\leq j\leq N.
\]
Define
\(
\widetilde h
=
\Pi_{E(3;3,1,1,1)}\circ F_0.
\)
Since $F_0$ is rational inner, its radial boundary values are unitary
almost everywhere on $\mathbb T$. Consequently, the corresponding
spectral data belong to 
\(
K\subseteq
\Gamma_{E(3;3,1,1,1)}
\)
almost everywhere on $\mathbb T$. Hence
\(
\widetilde h
\)
is a rational
$\Gamma_{E(3;3,1,1,1)}$-inner function.
Finally, at each interpolation node,
\[
\begin{aligned}
\widetilde h(z_j)
&=
\Pi_{E(3;3,1,1,1)}(F_0(z_j))\\
&=
\Pi_{E(3;3,1,1,1)}(f(z_j))\\
&=
h(z_j)\\
&=
\mathbf{x}_j.
\end{aligned}
\]
We now record the rational representation of the interpolating
function. Let
\(
\widetilde h
=
(x_1,\ldots,x_7)
\)
be a rational
$\Gamma_{E(3;3,1,1,1)}$-inner function. Since
$x_7$ is a scalar rational inner function, it is a finite Blaschke
product. Thus
\[
x_7(z)
=
\xi
\prod_{j=1}^{N}
\frac{z-\alpha_j}{1-\overline{\alpha_j}z},
\quad
|\xi|=1,\quad
\alpha_j\in\mathbb D.
\]
The boundary relations for
$\Gamma_{E(3;3,1,1,1)}$ give
\[
x_k(\zeta)
=
\overline{x_{7-k}(\zeta)}\,x_7(\zeta),
\quad
\zeta\in\mathbb T,
\quad
1\leq k\leq3.
\]
Consequently,
\[
\frac{x_k(\zeta)x_{7-k}(\zeta)}
{x_7(\zeta)}
=
|x_{7-k}(\zeta)|^2,
\quad
\zeta\in\mathbb T.
\]
The function on the left-hand side is rational and nonnegative on
$\mathbb T$. By the standard factorization theorem for nonnegative
rational functions on the unit circle (see \cite[p.~137]{Duren}), it
admits a factorization into finite Blaschke factors.
More precisely, for each $1\leq k\leq3$, there exist integers
$t_k,q_k\geq0$ with
\(
t_k+q_k=N,
\)
constants
\(
r_k>0,
|\eta_k|=|\delta_k|=1,
|\xi_{k,j}|=1,
\)
and finite Blaschke factors $f_{k,j}$ and $g_{k,j}$ such that
\[
x_k(z)
=
r_k\eta_k
\frac{
\displaystyle
\prod_{j=1}^{t_k}(\xi_{k,j}+z)
\prod_{j=1}^{q_k}f_{k,j}(z)
}{
\displaystyle
\prod_{j=1}^{N}(1-\overline{\alpha_j}z)
},\quad \text{
and}\quad
x_{7-k}(z)
=
r_k\delta_k
\frac{
\displaystyle
\prod_{j=1}^{t_k}(\xi_{k,j}+z)
\prod_{j=1}^{q_k}g_{k,j}(z)
}{
\displaystyle
\prod_{j=1}^{N}(1-\overline{\alpha_j}z)
}.
\]
The factors may be chosen in either of the two forms
\[
f_{k,j}(z)
=
\overline{\beta_{k,j}}z-1,
\qquad
g_{k,j}(z)
=
\beta_{k,j}-z,
\]
or
\[
f_{k,j}(z)
=
\beta_{k,j}-z,
\qquad
g_{k,j}(z)
=
\overline{\beta_{k,j}}z-1,
\]
where
\(
|\beta_{k,j}|=1.
\)
The unimodular constants satisfy the compatibility condition
\[
\eta_k\delta_k
\prod_{j=1}^{t_k}\xi_{k,j}
=
\xi.
\]
This proves the result.
\end{proof}
The remainder of this section is devoted to analogous interpolation
results for the domain $G_{E(3;2;1,2)}$. We state the results without
proof, as the proofs follow along the same lines as those given for
$G_{E(3;3;1,1,1)}$. The following theorem is obtained by applying the
characterization of $G_{E(3;2;1,2)}$ together with the classical
Nevanlinna--Pick interpolation theorem.

\begin{thm}\label{Interpolation 2}
Let $w_1,\ldots,w_N\in\mathbb D$ be distinct points, and let
\[
(x_{1,j},x_{2,j},x_{3,j},y_{1,j},y_{2,j})
\in G_{E(3;2;1,2)},
\quad 1\leq j\leq N.
\]
Suppose that there exists a holomorphic map
\(
h:\mathbb D\longrightarrow G_{E(3;2;1,2)}
\)
such that
\[
h(w_j)
=
(x_{1,j},x_{2,j},x_{3,j},y_{1,j},y_{2,j}),
\quad 1\leq j\leq N.
\]
Then, for every fixed $z\in\mathbb D$, the Pick matrix
\[
\left(
\frac{1-\overline{\widetilde w_j}\widetilde w_i}
     {1-\overline{w_j}w_i}
\right)_{i,j=1}^{N}
\]
is positive semidefinite, where
\(
\widetilde w_j
=
\Psi_3
\bigl(
z,(x_{1,j},x_{2,j},x_{3,j},y_{1,j},y_{2,j})
\bigr),
1\leq j\leq N,
\)
and
\(
(y_{1,j},y_{2,j})\in G_{E(2;1;2)}.
\)
\end{thm}
We next formulate the interpolation theorem for
$G_{E(3;2,1,2)}$ in terms of matrix-valued Schur functions. 
By \cite[Theorem 2.44]{apal1}, we have
\begin{equation}\label{ball inclusion E3212}
\Pi_{E(3;2;1,2)}
\bigl(B_{M_3(\mathbb C)}(0,1)\bigr)
\subseteq
G_{E(3;2;1,2)}.
\end{equation}
For $z=(z_1,z_2)\in\mathbb D^2$, put
\(
\pi_{E(3;3,1,1,1)}(z_1,z_2,z_3)
=
\Pi_{E(3;3,1,1,1)}
\left(\begin{smallmatrix}
z_1&0&0\\
0&z_1&0\\
0&0&z_2
\end{smallmatrix}\right).
\)
Analogously to Theorem~\ref{Interpolation in G_{E(3;3,1,1,1)}},
the following result provides a partial solution to the interpolation
problem by rational
$\Gamma_{E(3;2;1,2)}$-inner functions for collections of points in
$G_{E(3;2;1,2)}$ that are contained in
$\Pi_{E(3;2;1,2)}(B_{M_3(\mathbb C)}(0,1))$. Moreover, the proof
provides an explicit representation of the rational
$\Gamma_{E(3;2;1,2)}$-inner function solving the interpolation
problem.

\begin{thm}\label{Interpolation in G_{E(3;2;1,2)}}
Let $w_1,\ldots,w_N\in\mathbb D$ be distinct points and suppose that
\[
(x_{1,j},x_{2,j},x_{3,j},y_{1,j},y_{2,j})
\in
\Pi_{E(3;2;1,2)}
\bigl(B_{M_3(\mathbb C)}(0,1)\bigr),
\quad 1\leq j\leq N.
\]
Suppose further that there exists a holomorphic map
\(
h:\mathbb D\longrightarrow G_{E(3;2;1,2)}
\)
such that
\[
h(w_j)
=
(x_{1,j},x_{2,j},x_{3,j},y_{1,j},y_{2,j}),
\quad 1\leq j\leq N,
\]
and
\(
h(\mathbb D)
\subseteq
\Pi_{E(3;2;1,2)}
\bigl(B_{M_3(\mathbb C)}(0,1)\bigr).
\)
Then there exists a rational
$\Gamma_{E(3;2;1,2)}$-inner function
\(
\widetilde h:\mathbb D
\longrightarrow
\Gamma_{E(3;2;1,2)}
\)
such that
\[
\widetilde h(w_j)
=
(x_{1,j},x_{2,j},x_{3,j},y_{1,j},y_{2,j}),
\quad 1\leq j\leq N.
\]
Furthermore, $\widetilde h$ admits the following explicit
representation. There exist $n\in\mathbb N$, $\xi\in\mathbb T$, and
$\alpha_1,\ldots,\alpha_n\in\mathbb D$ such that, for each
$k=1,2$, there exist
\(
t_k,q_k\in\mathbb N\cup\{0\},
r_k\in\mathbb R,
\)
and
\(
\eta_k,\delta_k,\xi_{k,1},\ldots,\xi_{k,t_k}\in\mathbb T,
\beta_{k,1},\ldots,\beta_{k,q_k}\in\mathbb D,
\)
satisfying
\begin{equation}\label{parameters E3212}
t_k+q_k=n,
\qquad
\eta_k\delta_k
\prod_{i=1}^{t_k}\xi_{k,i}
=
\xi,
\qquad k=1,2.
\end{equation}
Then
\[
x_3(z)
=
\xi
\prod_{i=1}^{n}
\frac{z-\alpha_i}
     {1-\overline{\alpha_i}z},
\]
and
\begin{equation}\label{x_1_x_2_y_1_y_2 Representation}
\begin{aligned}
&x_1(z)
=
r_1\eta_1
\frac{
\displaystyle
\left(\prod_{i=1}^{t_1}(\xi_{1,i}+z)\right)
\left(\prod_{i=1}^{q_1}f_{1,i}(z)\right)
}{
\displaystyle
\prod_{i=1}^{n}(\overline{\alpha_i}z-1)
},\quad
x_2(z)
=
r_2\eta_2
\frac{
\displaystyle
\left(\prod_{i=1}^{t_2}(\xi_{2,i}+z)\right)
\left(\prod_{i=1}^{q_2}f_{2,i}(z)\right)
}{
\displaystyle
\prod_{i=1}^{n}(\overline{\alpha_i}z-1)
},
\\[2mm]
&y_1(z)
=
r_2\delta_2
\frac{
\displaystyle
\left(\prod_{i=1}^{t_2}(\xi_{2,i}+z)\right)
\left(\prod_{i=1}^{q_2}g_{2,i}(z)\right)
}{
\displaystyle
\prod_{i=1}^{n}(\overline{\alpha_i}z-1)
},\quad
y_2(z)
=
r_1\delta_1
\frac{
\displaystyle
\left(\prod_{i=1}^{t_1}(\xi_{1,i}+z)\right)
\left(\prod_{i=1}^{q_1}g_{1,i}(z)\right)
}{
\displaystyle
\prod_{i=1}^{n}(\overline{\alpha_i}z-1)
}.
\end{aligned}
\end{equation}
Here, for each $1\leq i\leq q_k$, either
\(
f_{k,i}(z)=\overline{\beta}_{k,i}z-1,
g_{k,i}(z)=\beta_{k,i}-z,
\)
or
\(
f_{k,i}(z)=\beta_{k,i}-z,
g_{k,i}(z)=\overline{\beta}_{k,i}z-1.
\)
The rational map
\(
\widetilde h
=
(x_1,x_2,x_3,y_1,y_2)
\)
is the required rational
$\Gamma_{E(3;2;1,2)}$-inner function interpolating the prescribed
data.
\end{thm}

\begin{rem}\label{Rem 5}
The results obtained in this section provide a partial solution to
the interpolation problem by rational
$\Gamma_{E(3;3;1,1,1)}$-inner and
$\Gamma_{E(3;2;1,2)}$-inner functions. More precisely, we construct
rational $\Gamma_{E(3;3;1,1,1)}$-inner functions and rational
$\Gamma_{E(3;2;1,2)}$-inner functions interpolating prescribed data
whenever the corresponding interpolation points belong to
\[
\Pi_{E(3;3;1,1,1)}
\bigl(B_{M_3(\mathbb C)}(0,1)\bigr)\quad \text{
and}\quad
\Pi_{E(3;2;1,2)}
\bigl(B_{M_3(\mathbb C)}(0,1)\bigr),
\]
respectively. Thus, the general interpolation problem by rational
$\Gamma_{E(3;3;1,1,1)}$-inner and rational
$\Gamma_{E(3;2;1,2)}$-inner functions for arbitrary interpolation
points in
$G_{E(3;3;1,1,1)}$ and
$G_{E(3;2;1,2)}$, respectively, remains open.
\end{rem}

\end{document}